\documentclass[a4paper,12pt]{amsart}
\usepackage{graphicx}
\usepackage{amsmath,amssymb,amsthm,amsfonts}
\usepackage{amssymb}
\usepackage[active]{srcltx}
\usepackage[colorlinks=true,urlcolor=blue,
citecolor=red,linkcolor=blue,linktocpage,pdfpagelabels,
bookmarksnumbered,bookmarksopen]{hyperref}
\usepackage{amsmath,amssymb,amsthm,amsfonts}
\usepackage{amssymb}
\usepackage[english]{babel}
\usepackage[colorlinks=true]{hyperref}

\newtheorem{thm}{Theorem}[section]

\newtheorem{lem}[thm]{Lemma}
\newtheorem{prop}[thm]{Proposition}
\newtheorem{rem}[thm]{Remark}

\newcommand{\R}{{\mathbb{R}}}
\newcommand{\N}{{\mathbb{N}}}

\everymath{\displaystyle}
\numberwithin{equation}{section}

\begin{document}


\parindent 0pc
\parskip 6pt
\overfullrule=0pt

\title{Classification of solutions to Hardy-Sobolev Doubly Critical Systems}

\author {Francesco Esposito$^{*}$, Rafael L\'opez-Soriano$^\#$, Berardino Sciunzi$^{*}$}

\date{\today}

\date{\today}

\address{* Dipartimento di Matematica e Informatica, UNICAL,
Ponte Pietro  Bucci 31B, 87036 Arcavacata di Rende, Cosenza, Italy.}

 \email{francesco.esposito@unical.it, sciunzi@mat.unical.it}
 
 \address{$\#$ Departamento de Análisis Matem\'atico,
Universidad de Granada, Campus Fuentenueva, 18071 Granada, Spain.}
 
 \email{ralopezs@ugr.es}


\thanks{F. Esposito and B. Sciunzi are partially supported by PRIN project 2017JPCAPN (Italy):
	{\em Qualitative and quantitative aspects of nonlinear PDEs.} Moreover they are members of INdAM. 
	R. López-Soriano is currently supported by the grant Juan de la Cierva Incorporación fellowship (JC2020-046123-I), funded by MCIN/AEI/10.13039/501100011033, and by the European Union Next Generation EU/PRTR. He is also partially supported by Grant PID2021-122122NB-I00 funded by MCIN/AEI/ 10.13039/501100011033 and by ``ERDF A way of making Europe''.}

\begin{abstract}
This work deals with a family of Hardy-Sobolev doubly critical system defined in $\mathbb{R}^n$. More precisely, we provide a classification of the positive solutions, whose expressions comprise multiplies of solutions of the decoupled scalar equation. Our strategy is based on the symmetry of the solutions, deduced via a suitable version of the moving planes technique for cooperative singular systems, joint with the study of the asymptotic behavior by using the Moser's iteration scheme.
\end{abstract}

\maketitle

\date{\today}

\section{Introduction}
We are concerned with the study of the doubly critical system:
\begin{equation}\tag{$\mathcal{S}^*$}\label{eq:doublecriticalsingularequbdd}
\begin{cases}
\displaystyle -\Delta u\,=\gamma \frac{u}{|x|^2} + u^{2^*-1}+ \nu \alpha u^{\alpha-1} v^\beta & \text{in}\quad\R^n \vspace{0.2cm} \\
\displaystyle -\Delta v\,=\gamma \frac{v}{|x|^2} + v^{2^*-1}+ \nu \beta u^\alpha v^{\beta-1} & \text{in}\quad\R^n \vspace{0.2cm} \\
u,v > 0 &  \text{in}\quad\R^n \setminus \{0\},
\end{cases}
\end{equation}
where $\gamma \in [0, \Lambda_n)$ with $\Lambda_n =\left( \frac{n-2}{2}\right)^2$ the best constant in the Hardy’s inequality for $n \geq 3$, $2^*=\dfrac{2n}{n-2}$ is the critical exponent in the Sobolev embedding, $\nu>0$ is the coupling parameter and $\alpha, \beta > 1$ are real parameters satisfying 
$$\alpha+\beta=2^*.$$
Along this article we address the analysis of the classification of positive solutions to \eqref{eq:doublecriticalsingularequbdd}. \\
\noindent In the case $\gamma = \nu = 0$, system \eqref{eq:doublecriticalsingularequbdd} reduces to the classical Sobolev critical equation
\begin{equation}\label{eq:Sobolev}
	\begin{cases}
		\displaystyle -\Delta u\,=  u^{2^*-1}  & \text{in}\quad\R^n\\
		u > 0 &  \text{in}\quad\R^n.
	\end{cases}
\end{equation}
It is well-known that the so-called \textit{Aubin-Talenti bubbles}
\begin{equation}\label{eq:Talentiane}
	\mathcal{V}_{\lambda, x_0} (x) := \left[\frac{\lambda \sqrt{n(n-2)}}{\lambda^2+|x-x_0|^2}\right]^{\frac{n-2}{2}}
\end{equation}
solve \eqref{eq:Sobolev}, where  $\lambda> 0$ and it is centered at $x_0 \in \R^n$. Moreover, they realize the equality in the sharp Sobolev inequality in $\R^n$. In \cite{GNN2}, the authors proved that any positive solution $u$ to \eqref{eq:Sobolev}, such that $u(x) = O\left(1/|x|^m\right),$ at infinity for $m>0$, is radially symmetric and decreasing about some point in $\R^n$. In the proof the authors used a refinement of the celebrated moving plane procedure developed by themselves in a previous paper. Later on, Caffarelli, Gidas and Spruck in \cite{CGS} classified all the solutions to \eqref{eq:Sobolev}. Making use of the Kelvin transform, the authors showed that the moving plane procedure can start.
Finally, they showed that all the solutions $u \in H^1_{loc}(\R^n)$ to \eqref{eq:Sobolev} are given by the Aubin-Talenti bubbles \eqref{eq:Talentiane}, hence solutions are unique up to rescaling. Recently, these results were also generalized in the case of critical equations involving the $p$-Laplacian and the Finsler anisotropic operator, where the use of the Kelvin transform is not-allowed, see e.g.~\cite{CatMonRon, CFR, Ou, Sciu16, Vet, Vet2}.

When $\gamma \neq 0$ and $\nu = 0$,  system \eqref{eq:doublecriticalsingularequbdd} becomes the so-called Hardy-Sobolev doubly critical equation
\begin{equation}\label{eq:terracini}
	\begin{cases}
	-\Delta u \,=\gamma \frac{u}{|x|^2} + u^{2^*-1} \quad & \text{in}\quad\R^n \\
		u > 0 &  \text{in}\quad\R^n \setminus \{0\}.
 \end{cases}
\end{equation}
Terracini in \cite{terracini},  by means of variational arguments and the concentration compactness principle, showed the existence of solutions to \eqref{eq:terracini}, actually a more general one. Moreover, by using the Kelvin transformation and a fine use of the moving plane method, she proved that any solution to \eqref{eq:terracini} is radially symmetric about the origin. Finally, thanks to a detailed ODE's analysis, she gave a complete classification of the solutions to \eqref{eq:terracini}, given by
\begin{equation}\label{eq:terraciniane}
	\mathcal{U}_\mu(x)=\mu^{\frac{2-n}{2}}\mathcal{U}\left(\frac{x}{\mu}\right) \qquad \mbox{ with } \qquad \mathcal{U}(x)= \dfrac{A(n,\gamma)}{|x|^{\tau}\left(1+|x|^{2-\frac{4\tau}{n-2}}\right)^{\frac{n-2}{2}}},
\end{equation}
where 
\begin{equation}\label{eq:agamma}
	\tau:=\frac{n-2}{2}-\sqrt{\left( \frac{n-2}{2}\right)^2-\gamma}, \qquad A(n,\gamma):=\left(\frac{n(n-2-2\tau)^2}{n-2}\right)^{\frac{n-2}{4}},
\end{equation}
and $\mu>0$ is a scaling factor. Obviously, $\mathcal{U}_\mu=\mathcal{V}_{\mu,0}$ if $\gamma=0$.

The case of $p$-Lapalce operator was firstly treated in \cite{abdellaoui}, where the authors carried out a very fine ODE's analysis. The radial symmetry of the solutions there was an assumption. Later on, in \cite{OSV} the authors showed that all the positive solutions to $p$-Laplace doubly critical equation are radial (and radially decreasing) about the origin. Once the radial symmetry of the solution is proved it is easy to derive the associated ordinary differential equation fulfilled by the solution $u = u(r)$ and, hence, to apply the results in \cite{abdellaoui}. 

>From now on we focus our attention to the case of systems. Nonlinear Schr\"odinger problems, like the Gross–Pitaevskii type systems, have a strong connection with some physical phenomena. Such  problem appears in the study of Hartree–Fock theory for double condensates, that is a binary mixture Bose–Einstein condensates in two different hyperfine states which overlap in space, see \cite{Esry, Frantz} for further details. That type of systems arises also in nonlinear optics. Actually, it allows one to study the propagation of pulses in birefringent optical fibers and the beam in Kerr-like photorefractive media, see \cite{Akhmediev,Kivshar} and references therein.
In particular, solitary-wave solutions to the coupled Gross–Pitaevski equations satisfies the problem
\begin{equation}\label{eq:pBSsystem}
\left\{\begin{array}{ll}
\displaystyle{-\Delta u + V(x) u= \mu u^{2p-1} + \nu u^{p-1} v^{p}}  &\text{in }\mathbb{R}^n\vspace{.3cm}\\
-\Delta v + V(x) v = \mu v^{2p-1} + \nu u^{p}v^{p-1} &\text{in }\mathbb{R}^n,
\end{array}\right.
\end{equation}
where $V$ is the system potential and $1<p\le\frac{2^*}{2}$. This problem is typically known as the Bose–Einstein condensate system. For the subcritical regime, we refer to \cite{ambrosetti,bartsch2,lin1,sirakov, nicola2} for some results concerning existence and multiplicity of solutions under different assumptions on $V$ and $\nu$.

Concerning the critical case $p = \frac{2^*}{2}$, if $V$ is  non-zero constant, then  system \eqref{eq:pBSsystem} admits only the trivial solution $(0,0)$. This result follows from a proper application of the Pohozaev-type identity. For $V=0$, in \cite{Wang} the authors showed the uniqueness of the ground states, under suitable assumptions on the parameters of the generalized system; whereas in the paper \cite{pistoia} the competitive setting is considered, i.e. $\nu<0$, deducing that the system admits infinitely many fully nontrivial solutions, which are not conformally equivalent.

As a non-constant potential, from now on we shall consider the aforementioned Hardy-type one $V=-\frac{\gamma}{|x|^2}$. Under that choice, the problem \eqref{eq:doublecriticalsingularequbdd} can be also seen  as an extension of \eqref{eq:pBSsystem}. The mathematical interest in such system lies in their double criticality, since both the exponent of the nonlinearities and the singularities share the same order of homogeneity as the Laplacian. Moreover, inverse square potentials arise in some physical prototypes, such as nonrelativistic quantum mechanics, molecular physics or combustion models.

Doubly critical problems has attracted attention in recent years. In the pioneer work \cite{AFP}, a general Hardy-Sobolev type system is studied by means of variational techniques. In the cooperative regime $\nu>0$, the existence of ground and bound states are obtained depending on $\alpha, \beta$, $n$ and a potential function $h$ arising in the coupling term. Recently, that kind of results were extended in \cite{EduRafaAle} by using similar strategies.  Such doubly critical system is widely analyzed in \cite{ChenZou3}. Actually, the non-existence of ground states for the competitive case is indeed proved.

The aim of this paper is to classify all the positive solutions to  problem \eqref{eq:doublecriticalsingularequbdd}, via the study of symmetry and monotonicity properties in the same spirit of the aforementioned papers \cite{CGS, terracini}. 

Hence, we state one of the main results of our paper:
\begin{thm}\label{thm:main2}
	Let $(u,v) \in D^{1,2}(\mathbb{R}^n) \times D^{1,2}(\mathbb{R}^n)$
	be a solution to \eqref{eq:doublecriticalsingularequbdd}. 
	%
	%
	Then,
	\begin{equation}\label{eq:sinchronizedSol}
		(u,v)=\left(c_1 \, \mathcal{U}_{\mu_0}, c_2 \, \mathcal{U}_{\mu_0}\right),
	\end{equation}
	with $\mathcal{U}_\mu$ introduced in \eqref{eq:terraciniane}, $\mu_0>0$ and $c_1,c_2$ are any positive constants satisfying the system
	\begin{equation}\label{eq:systemconstants}
		\left\{\begin{array}{ll}
			\displaystyle{c_1^{2^*-2}+\nu \alpha c_1^{\alpha-2}c_2^\beta}=1  & \vspace{.3cm}\\
			\displaystyle{c_2^{2^*-2}+\nu \beta c_1^{\alpha}c_2^{\beta-2}=1}. &
		\end{array}\right.
	\end{equation}
\end{thm}
\begin{rem}
This result holds (and is new)  also in the case  $\gamma=0$.  Under that assumption the explicit solutions in \eqref{eq:terraciniane} reduce to those of \eqref{eq:Talentiane}. See \cite{ChenLi} for a result related to such a case.
\end{rem}

\begin{rem}
	As a consequence of Theorem \ref{thm:main2}, we can give an explicit expression of the solution. Observe that the system \eqref{eq:systemconstants} might admit several solutions depending on the values $n, \nu, \alpha, \beta$. For instance, if $n=3$, $\nu=1$, $\alpha=3=\beta$, there exist three synchronized solutions; whereas if $n=4$, $\nu=1$, $\alpha=2=\beta$ there exists one synchronized solution. Concerning the uniqueness of synchronized solutions we refer \cite{Wang} for more details.
\end{rem}

\begin{rem}\label{rem:number}
	From a pure mathematical point of view, let us emphasize that, throughout the paper, a very deep and crucial issue is represented by those cases in which either $\alpha<2$, $\beta<2$ or both. In such a case, in fact, the coupling term is non locally Lipschitz continuous. Note that this necessarily occurs if $n>4$. In low dimension, however, all the situations are possible.
\end{rem}

\noindent A first important step in our strategy  is the proof of the fact that
 the solutions to \eqref{eq:doublecriticalsingularequbdd} are  radially symmetric about the origin. 
 Actually we shall provide a more general result  and, in particular, holds without requiring global energy information.

\begin{thm}\label{thm:main1}
	Let $(u,v) \in \left( H^1_{loc}(\mathbb{R}^n) \cap C(\R^n \setminus \{0\}) \right)\times \left(H^1_{loc}(\mathbb{R}^n) \cap C(\R^n \setminus \{0\}) \right)$ be solution to 
	\begin{equation}\tag{$\mathcal{S}^*_{\gamma_{1,2}}$}\label{eq:doublecriticalsingularequbdddiffgamma}
		\begin{cases}
			\displaystyle -\Delta u\,=\gamma_1 \frac{u}{|x|^2} + u^{2^*-1}+ \nu \alpha u^{\alpha-1} v^\beta & \text{in}\quad\R^n \vspace{0.2cm} \\
			\displaystyle -\Delta v\,=\gamma_2 \frac{v}{|x|^2} + v^{2^*-1}+ \nu \beta u^\alpha v^{\beta-1} & \text{in}\quad\R^n   \\
			u,v > 0 &  \text{in}\quad\R^n \setminus \{0\},
		\end{cases}
	\end{equation}
	where $\gamma_1,\gamma_2 \in [0,\Lambda_n)$. Then $(u,v)$ is radially symmetric about the origin.
\end{thm}

The proof of the radial symmetry of the solutions is based on a fine adaptation of the moving plane procedure. The technique goes back to the seminal works of Alexandrov and Serrin \cite{A,serrin} and the well-known contributions by Berestycki-Nirenberg \cite{BN} and Gidas-Ni-Nirenberg \cite{GNN}. Originally, the technique was introduce to be performed in general domains providing partial monotonicity results near the boundary and symmetry for convex and symmetric domains. Regarding elliptic systems, the moving plane technique was adapted by Troy in \cite{troy}, where the cooperative case is analyzed. The procedure was also applied for semilinear systems in the half-space in \cite{Dan} and in the whole space by Busca and Sirakov in the work \cite{busca}. We refer the reader to \cite{pacef,pacef2, defig2,esposito, Farina, FarSciuSo, mitidieri1,RZ, nicola1} for other interesting contribution about elliptic systems in bounded or unbounded domains.

\
\noindent For the reader's convenience we describe the strategy of our proofs that turn out to be tricky somehow.
\subsection{The symmetry result}
The proof of Theorem \ref{thm:main1}, as recalled above, is based on the moving plane technique. Unfortunately the adaptation of the technique is not straightforward since we work in unbounded domains and the coupling term, in general, is not locally Lipschitz continuous.
\noindent We overcame such a difficulty studying a suitable translated problem. To obtain symmetry in the $x_1$-direction move the Hardy potential to
$$x_0=(0,x'_{0}) \in \R^n   \setminus \{0\}.$$ 
 Then we apply the Kelvin transformation. The translation argument allow us to deal with the presence of the Hardy potential. The problem is not invariant under this procedure but the equation that arises is not so bad and we succeed in the adaptation of the moving plane procedure. 
\subsection{The asymptotic analysis}
Once we know that the solutions are radially symmetric, we use a suitable transformation 
\begin{equation}\nonumber
	(u_\tau(x), v_\tau(x)) := \left(|x|^\tau u(x), |x|^\tau v(x) \right)\,,
\end{equation}
 for the right choice of $\tau>0$. The Moser iteration scheme, applyed to the transformed equation provides a first asymptotic information. Then the study of the associated ODE (together with the Kelvin transform)  allows us to deduce the precise asymptotic information at zero and at infinity. 

\subsection{The classification result}
Once we know the exact behavior of the solutions, we exploit the standard change of variable
\begin{equation} \nonumber
	\begin{split}
		y_u(t):=r^\delta u(r) \quad &\text{and} \quad y_v(t):=r^\delta v(r),\\
	\end{split}
\end{equation}
where $t:=\log r$ with $r>0$ and $\delta=\frac{n-2}{2}$. We prove that
\begin{center}
there exists $C>0$ such that $y_u= C y_v$
\end{center}
where $C$ is a zero of the function 
\begin{equation}\nonumber
	f(s)=s^{2^*-2}+\nu \alpha s^{\alpha-2}-1-\nu\beta s^{\alpha}.
\end{equation}
The number of roots of $f$ is equivalent to the solutions of the system \eqref{eq:systemconstants}, then such quantity gives us the number of synchronized solutions. Although in this final issue we are reduced to deal with an ODE analysis, the proof is  no longer standard. To the best of our knowledge, only the case of a single root of $f$ has been treated in the literature. Precisely, the hardest part is the case when the function $f(\cdot)$ has more than one zero, a possible issue, see Remark~\ref{rem:number}. When the proportionality of components $(y_u,y_v)$ is guaranteed, one can conclude the classification result by direct computation.

\section{Radial symmetry of the solutions, proof of Theorem \ref{thm:main1}} \label{sec2}

The aim of this section is to prove that any solution to \eqref{eq:doublecriticalsingularequbdddiffgamma} is radially symmetric about the origin. 
We shall therefore consider positive continuous (far from the origin) solutions 
 $(u,v)\in H^1_{loc}(\R^n)\times H^1_{loc}(\R^n)$ such that
\begin{equation}\label{weaksolutions}
	\begin{split}
		\int_{\mathbb{R}^n} \langle \nabla u, \nabla
		\Phi \rangle\,dx\,=\,\gamma_1 \int_{\R^n} \frac{u}{|x|^2} \Phi\,dx +\int_{\mathbb{R}^n} {u}^{2^*-1}\Phi\,dx + \nu \alpha\int_{\mathbb{R}^n} {u}^{\alpha-1} {v}^\beta \Phi\,dx,& \vspace{0.2cm}\\
		\displaystyle \int_{\mathbb{R}^n} \langle \nabla {v}, \nabla
		\Psi \rangle \,dx\,=\gamma_2 \int_{\R^n} \frac{ v}{|x|^2} \Psi\,dx +\,\int_{\mathbb{R}^n} {v}^{2^*-1}\Psi\,dx + \nu \beta \int_{\mathbb{R}^n} {u}^\alpha {v}^{\beta-1}	\Psi \,dx,&
	\end{split}
\end{equation}
for any $\Phi, \Psi\in C^{1}_c(\mathbb{R}^n)$ and $\gamma_1, \gamma_2 \in [0, \Lambda_n)$.
To prove Theorem \ref{thm:main1}, we need to fix some notations. For any real number $\lambda$ we set
\begin{equation}\label{eq:sn2} \nonumber
\Sigma_\lambda=\{x\in \R^n:x_1 <\lambda\} 
\end{equation}
\begin{equation}\label{eq:sn3} \nonumber
x_\lambda= R_\lambda(x)=(2\lambda-x_1,x_2,\ldots,x_n)
\end{equation}
which is the reflection through the hyperplane $T_\lambda :=\{ x_1=
\lambda\}$. Moreover, given any function $w$,  we will set 
\begin{equation}\label{eq:sn4} \nonumber  
w_{\lambda} := w \circ R_{\lambda},
\end{equation}
namely the reflected function. A crucial ingredient in our proof is the use of a translation argument. 
We fix
$$x_0=(0,x'_{0}) \in \R^n \setminus  \{0\}$$
and we assume, by translation, that $(u,v)$ solves
\begin{equation}\tag{$\mathcal{S}^*_{x_0}$}\label{eq:doublecriticalsingularequbddshift}
	\begin{cases}
		\displaystyle -\Delta u\,=\gamma_1 \frac{u}{|x-x_0|^2} + u^{2^*-1}+ \nu \alpha u^{\alpha-1} v^\beta & \text{in}\quad\R^n \vspace{0.2cm} \\
		\displaystyle -\Delta v\,=\gamma_2 \frac{v}{|x-x_0|^2} + v^{2^*-1}+ \nu \beta u^\alpha v^{\beta-1} & \text{in}\quad\R^n\\
		u,v > 0 &  \text{in}\quad\R^n \setminus \{x_0\}.
	\end{cases}
\end{equation}
This will allows us to take full advantage of the Kelvin transformation. In fact we set 
  $K: \R^n \setminus \{0\} \longrightarrow \R^n  \setminus \{0\} $ defined by 
$$K = K(x):=\frac{x}{|x|^2}.$$
Such a transformation is a well-known tool and, given any $(u,v)$ solution to \eqref{eq:doublecriticalsingularequbddshift}, its Kelvin transform  is defined as
\begin{equation}\label{E:kelv} \left(\hat{u}(x),\hat{v}(x)\right):=\left( \frac{1}{|x|^{n-2}} u\left(\frac{x}{|x|^2}\right), \frac{1}{|x|^{n-2}} v\left(\frac{x}{|x|^2}\right)\right) \quad
	x \in \R^n \setminus \{0,x_0\}.
\end{equation}
By direct computation it follows that $(\hat{u},\hat{v})$ weakly
satisfies 
\begin{equation}\tag{$\hat{\mathcal{S}}^*_{x_0}$}\label{eq:doublecriticalsingularequbddshiftKelv}
	\begin{cases}
		\displaystyle -\Delta \hat u\,=\gamma_1 \frac{\hat u}{f_{x_0}(x)} + \hat{u}^{2^*-1}+ \nu \alpha \hat{u}^{\alpha-1} \hat{v}^\beta & \text{in}\quad\R^n \vspace{0.2cm} \\
		\displaystyle -\Delta \hat{v}\,=\gamma_2 \frac{\hat{v}}{f_{x_0}(x)} + \hat{v}^{2^*-1}+ \nu \beta \hat{u}^\alpha \hat{v}^{\beta-1} & \text{in}\quad\R^n \vspace{0.2cm} \\
		\hat u,\hat v > 0 &  \text{in}\quad\R^n \setminus \{0,x_0\},
	\end{cases}
\end{equation}
where
\begin{equation}\label{eq:Hardyfunction}
	f_{x_0}(x):=\big|x-x_0|x|^2\big|^2.
\end{equation}
 Having in mind the last definition, we prove the following key lemma that will be used in the proof of the symmetry result.

\begin{lem} \label{lem:monotonicity} 
	Let $x_0 \in \R^n \setminus\{0\}$ be any fixed point such that $x_0=(0,x'_0)=(0,x_{0,2},\dots,x_{0,n})$. Then the function $f_{x_0}$ is increasing in the $x_1$-direction for any $x \in \R^n \setminus \Sigma_0$, decreasing in in the $x_1$-direction for any $x \in \Sigma_0$, and symmetric with respect to $T_0$.
\end{lem}

\begin{proof}
	 It is immediate to have that
	\begin{equation}
		\begin{split}
			f_{x_0}(x)&=\big|x-x_0|x|^2\big|^2 = \langle x-x_0|x|^2, x-x_0|x|^2\rangle\\
			&=|x|^2-2|x|^2\langle x, x_0 \rangle + |x_0|^2|x|^4\\
			&=|x|^2\left(1-2\langle x, x_0 \rangle + |x_0|^2|x|^2\right),
		\end{split}
	\end{equation}
	for every $x \in \R^n$. We observe that thanks to our assumptions, the term $\langle x,x_0\rangle$ does not depend on $x_1$. Hence, thanks to Schwarz inequality we deduce that for any $x \in \R^n \setminus \Sigma_0$ it holds
	\begin{equation}
	\begin{split}
		\frac{\partial f_{x_0}}{\partial x_1}(x)&=2x_1\left(1-2\langle x,x_0\rangle + 2|x_0|^2|x|^2\right)\\
		&\geq 2x_1\left(1-2|x| \cdot |x_0| + 2|x_0|^2|x|^2\right)\\
		&=2x_1\left[\left(1-|x_0| \cdot |x|\right)^2 + |x_0|^2 \cdot |x|^2\right] \geq 0.
	\end{split}
	\end{equation}
Analogously, we get $\frac{\partial f_{x_0}}{\partial x_1}(x) \leq 0$ for any $x \in \Sigma_0$.

\end{proof}

We shall prove a symmetry result regarding
 any solution to \eqref{eq:doublecriticalsingularequbddshiftKelv}, that translates into a symmetry result for the original problem  \eqref{eq:doublecriticalsingularequbdddiffgamma}.
 In particular, the couple $(\hat{u}, \hat{v})$  satisfies
\begin{equation}\label{eq:weakKelv3}
	\begin{array}{lr}
		\displaystyle \int_{\mathbb{R}^n} \langle \nabla \hat{u}, \nabla
		\Phi \rangle \,dx\,=\,\gamma_1 \int_{\R^n} \frac{\hat u}{f_{x_0}(x)} \Phi \,dx + \int_{\mathbb{R}^n} \hat{u}^{2^*-1}\Phi\,dx + \nu \alpha\int_{\mathbb{R}^n} \hat{u}^{\alpha-1} \hat{v}^\beta	\Phi \,dx, \\
		\\
		\displaystyle \int_{\mathbb{R}^n} \langle \nabla \hat{v}, \nabla
		\Psi \rangle \,dx\,=\,\gamma_2 \int_{\R^n} \frac{\hat v}{f_{x_0}(x)} \Psi\,dx + \int_{\mathbb{R}^n} \hat{v}^{2^*-1}\Psi\,dx + \nu \beta\int_{\mathbb{R}^n} \hat{u}^\alpha \hat{v}^{\beta-1}	\Psi \,dx,
	\end{array}
\end{equation}\\
for any $\Phi, \Psi\in C^{1}_c(\mathbb{R}^n\setminus \{0,x_0\})$, while reflecting through $T_\lambda$ we deduce that   $(\hat{u}_\lambda, \hat{v}_\lambda)$ satisfies
\begin{equation}\label{eq:weakKelv4}
	\begin{array}{lr}
		\displaystyle \int_{\mathbb{R}^n} \langle \nabla \hat{u}_\lambda, \nabla
		\Phi \rangle\,dx\,=\,\gamma_1 \int_{\R^n} \frac{\hat u_\lambda}{f_{x_0}(x_\lambda)} \Phi\,dx +\int_{\mathbb{R}^n} \hat{u}_\lambda^{2^*-1}\Phi\,dx + \nu \alpha\int_{\mathbb{R}^n} \hat{u}_\lambda^{\alpha-1} \hat{v}_\lambda^\beta	\Phi \,dx, \\
		
		\\
		\displaystyle \int_{\mathbb{R}^n} \langle \nabla \hat{v}_\lambda, \nabla
		\Psi \rangle \,dx\,=\gamma_2 \int_{\R^n} \frac{\hat v_\lambda}{f_{x_0}(x_\lambda)} \Psi\,dx +\,\int_{\mathbb{R}^n} \hat{v}_\lambda^{2^*-1}\Psi\,dx + \nu \beta \int_{\mathbb{R}^n} \hat{u}_\lambda^\alpha \hat{v}_\lambda^{\beta-1}	\Psi \,dx,
	\end{array}
\end{equation}
for any $\Phi, \Psi\in C^{1}_c(\mathbb{R}^n\setminus R_\lambda(\{0 , x_0\}))$.

\noindent A crucial point in all the paper, as recalled in the introduction, is represented by the fact that the coupling term may be not Lipschitz continuous at zero. To face this difficulty, we will also use the following:

\begin{lem} \label{lem:positivity}
	Let $(\hat{u},\hat{v})$ be a solution to \eqref{eq:doublecriticalsingularequbddshiftKelv}. Then, there exists $\bar R >0$ and $\mu>0$ such that 
	\begin{equation}
		\hat u(x) \geq \mu > 0 \quad \text{and} \quad \hat v (x) \geq \mu > 0 \quad \text{in } B_{\bar R}.
	\end{equation}
\end{lem}

\begin{proof}
	First of all, we note that $(\hat{u}, \hat{v})$ satisfies \eqref{eq:weakKelv3}.  Borrowing an idea contained in \cite{EFS}, we point out that, for every $ \varepsilon>0$, we can find a function $\psi_\varepsilon \in C^{0,1}(\R^N, [0,1])$
	such that
	$$\int_{\Sigma_\lambda} |\nabla \psi_\varepsilon|^2 <  \varepsilon$$
	and $\psi_\varepsilon = 0$ in an open neighborhood
	$\mathcal{B_{\varepsilon}}$ of $0$, with $\mathcal{B_{\varepsilon}} \subset
	\R^n$.  We note that there exists $\bar R> 0$ such that $x_0 \not \in B_{\bar{R}}(0)$. Since $\hat u > 0$ in $B_{\bar{R}}(0)\setminus \{0\}$, there exists $\mu > 0$ such that $\hat u > \mu > 0$. In the same way  $\hat v > \mu > 0$ on $\partial B_{\bar{R}}(0)$.  Hence, setting
	\begin{equation}\label{eq:cutoff}
		\Phi:= (\mu- \hat u)^+ \psi_\varepsilon^2 \chi_{B_{\bar R}(0)}\quad \text{and} \quad \Psi:= (\mu- \hat v)^+ \psi_\varepsilon^2 \chi_{B_{\bar R}(0)},
	\end{equation}
	 one can check that $\Phi, \Psi \in D^{1,2}(\R^n)$, and by density arguments we can choose these as test functions in \eqref{eq:weakKelv3}. For the reader convenience we make the computations just for the first equation of \eqref{eq:weakKelv3}. Hence. we are able to deduce that
		\begin{equation}\label{eq:weakKelv1bis}
		\begin{split}
			\displaystyle - \int_{B_{\bar R}(0)} |\nabla	(\mu - \hat u)^+|^2 \psi_\varepsilon^2\,dx\, + 2 \int_{B_{\bar R}(0)} \langle \nabla \hat u, \nabla \psi_\varepsilon \rangle (\mu - \hat u)^+ \psi_\varepsilon \geq 0.
		\end{split}
	\end{equation}
 Using the Young's inequality, we are able to get that
  	\begin{equation}\label{eq:weakKelv1tris}
  	\begin{split}
  		\displaystyle \int_{B_{\bar R}(0)} |\nabla	(\mu - \hat u)^+|^2 \psi_\varepsilon^2 \, dx & \leq  2 \int_{B_{\bar R}(0)} \langle \nabla \hat u, \nabla \psi_\varepsilon \rangle (\mu - \hat u)^+ \psi_\varepsilon \, dx\\
  		& \leq \frac{1}{2}\int_{B_{\bar R}(0)} |\nabla (\mu - \hat u)^+|^2 \psi_\varepsilon^2 \, dx + 2\int_{B_{\bar R}(0)} |\nabla \psi_\varepsilon|^2 [(\mu - \hat u)^+]^2 \,dx.
  	\end{split}
  \end{equation}
Finally, we have that 
  \begin{equation}\label{eq:weakKelv1poker}
  	\int_{B_{\bar R}(0)} |\nabla	(\mu - \hat u)^+|^2 \psi_\varepsilon^2 \, dx  \leq 4 \mathcal{C}  \int_{B_{\bar R}(0)} |\nabla \psi_\varepsilon|^2  \,dx \leq 4 \mathcal{C} \varepsilon.
  \end{equation}
Passing to the limit for $\varepsilon$ that goes to $0$ we get the thesis for $\hat u$. Arguing in a similar fashion, we obtain the same result for $\hat v$.

\end{proof}
The translation argument that we introduced allow us to deduce the following:
\begin{lem} \label{lem:asBehaviour}
	Let $(\hat{u},\hat{v})$ be a solution to \eqref{eq:doublecriticalsingularequbddshiftKelv}. Then, there exist $c_{\hat u}, C_{\hat u} >0$, $c_{\hat v},C_{\hat v}>0$ and $\tilde R := \tilde R (x_0) > 0$ such that
	\begin{equation}\label{eq:bbehaviour}
		\frac{c_u}{|x|^{n-2}} \leq \hat u(x) \leq \frac{C_u}{|x|^{n-2}} \quad \text{and} \quad \frac{c_v}{|x|^{n-2}} \leq \hat v (x) \leq \frac{C_v}{|x|^{n-2}} \quad \text{in } \R^n \setminus B_{\tilde R}(0). 
	\end{equation}
\end{lem}

\begin{proof}
Since $x_0=(0,x_0') \in \R^n$ is fixed, then there exists $\delta > 0$ such that $x_0 \not \in B_{\delta}(0)$. By our assumptions $u,v \in  C(\R^n \setminus \{x_0\})$, and hence we deduce that $u,v \in C(B_\delta(0))$. We fix $\tilde R > \frac{1}{\delta}$ in such a way that by the definition of Kelvin transformation \eqref{E:kelv}  we easily deduce \eqref{eq:bbehaviour}.

\end{proof}

Now, we are ready to prove that any positive solution to \eqref{eq:doublecriticalsingularequbddshiftKelv} is symmetric in the $x_1$ direction. 

\begin{proof}[Proof of Theorem \ref{thm:main1}]
	Let us consider 
	\begin{equation}\label{eq:mov} \nonumber
	\begin{split}
	\xi_{\lambda}(x):=\hat{u}(x)-\hat{u}_\lambda (x)= \hat{u}(x)- \hat{u}(x_\lambda),\\
	\zeta_{\lambda}(x):=\hat{v}(x)-\hat{v}_\lambda (x)= \hat{v}(x) - \hat{v}(x_\lambda).
	\end{split}
	\end{equation}

	We recall that we are working with the weak formulations \eqref{eq:weakKelv3} and \eqref{eq:weakKelv4}.
	By Lemma \ref{lem:asBehaviour} we deduce that $\vert \hat{u}(x)
	\vert \le C_u \vert x \vert^{2-n}$ and $\vert \hat{v}(x)
	\vert \le C_v \vert x \vert^{2-n} $ and for every $x \in \R^n$ such that
	$\vert x \vert \ge \tilde R$, where $C_u, C_v$ and $\tilde R$ are positive constants
	(depending on $u$ and $v$).  In particular, for every $ \lambda < - \tilde R < 0$, we have
	\begin{equation}\label{E:kelvinp} \nonumber
	\hat{u}, \hat{v}\in L^{2^*}(\Sigma_\lambda)\cap L^{\infty}(\Sigma_\lambda) \cap C^0 (\overline{\Sigma_\lambda}) \,.
	\end{equation}

	In order to complete the proof of our result, we split the proof into three steps.

    \noindent \emph{Step 1: There exists $M>0$} large such that $\hat{u} \leq \hat{u}_\lambda$ and $\hat{v} \leq \hat{v}_\lambda$ in $\Sigma_\lambda\setminus R_{\lambda}(\{0 , x_0\})$, for all $\lambda< -M$.
		We immediately see that $\xi_\lambda^+, \zeta_\lambda^+ \in L^{2^*}(\Sigma_\lambda),$
		since 
		$$0\leq \xi_\lambda^+\leq \hat{u} \in L^{2^*}(\Sigma_\lambda) \ \ \text{ and } \ \ 0\leq \zeta_\lambda^+\leq \hat{v} \in L^{2^*}(\Sigma_\lambda).$$ 
		We point out that, for every $ \varepsilon>0$, we can find a function $\psi_\varepsilon \in C^{0,1}(\R^N, [0,1])$
		such that
		$$\int_{\Sigma_\lambda} |\nabla \psi_\varepsilon|^2 <  \varepsilon$$
		and $\psi_\varepsilon = 0$ in an open neighborhood
		$\mathcal{B_{\varepsilon}}$ of $R_{\lambda}(\{0 , x_0\})$, with $\mathcal{B_{\varepsilon}} \subset
		\Sigma_{\lambda}$.
		
		Fix $ R_0>0$ such that $R_\lambda(\{0,x_0\}) \in
		B_{R_0} $ and, for every $ R > R_0$, let $\eta_R$ be a standard
		cut off function such that $ 0 \le \eta_R \le 1 $ on $ \R^n$,
		$\eta_R=1$ in $B_R$, $\eta_R=0$ outside $B_{2R}$ with
		$|\nabla\eta_R|\leq 2/R,$ and consider
		$$\Phi\,:= \begin{cases}
		\, \xi_\lambda^+\psi_\varepsilon^2\eta_R^2 \, & \text{in}\quad\Sigma_\lambda, \\
		0 &  \text{in}\quad \R^n \setminus  \Sigma_\lambda
		\end{cases} \quad \text{and} \quad
		\Psi\,:= \begin{cases}
		\, \zeta_\lambda^+\psi_\varepsilon^2\eta_R^2 \, & \text{in}\quad\Sigma_\lambda, \\
		0 &  \text{in}\quad \R^n \setminus  \Sigma_\lambda.
		\end{cases}$$
		Now, it is easy to infer that $ \Phi, \Psi \in C^{0,1}_c(\R^n)$ with $\text{supp}(\Phi)$ and  $\text{supp}(\Psi)$ contained in $\overline {\Sigma_{\lambda} \cap B_{2R}} \setminus R_\lambda(\{0,x_0\})$ and
		\begin{equation}\label{E:gradvarphi-intero}
		\nabla \Phi = \psi_\varepsilon^2 \eta_R^2 \nabla \xi^+_\lambda +
		2 \xi_\lambda^+  (\psi_\varepsilon^2 \eta_R \nabla \eta_R + \psi_\varepsilon \eta_R^2  \nabla  \psi_\varepsilon),
		\end{equation}
		\begin{equation}\label{E:gradpsi-intero}
		\nabla \Psi = \psi_\varepsilon^2 \eta_R^2 \nabla \zeta^+_\lambda +
		2 \zeta_\lambda^+ (\psi_\varepsilon^2 \eta_R \nabla \eta_R + \psi_\varepsilon \eta_R^2  \nabla  \psi_\varepsilon).
		\end{equation}
		Therefore, by a standard density argument, we can plug $\Phi$ and $\Psi$ as  test functions in \eqref{eq:weakKelv3} and  \eqref{eq:weakKelv4} respectively, so that, subtracting we get
		\begin{equation}\label{eq:diffu1}
		\begin{split}
		\int_{\Sigma_\lambda}|\nabla \xi^+_\lambda|^2 \psi_\varepsilon^2\eta_R^2\,dx=&
		-2\int_{\Sigma_\lambda}\langle \nabla \xi^+_\lambda, \nabla \psi_\varepsilon \rangle
		\xi_\lambda^+ \psi_\varepsilon\eta_R^2\,dx
		-2\int_{\Sigma_\lambda} \langle \nabla \xi^+_\lambda, \nabla \eta_R \rangle \xi_\lambda^+ \eta_R \psi_\varepsilon^2\,dx\\
		&+\gamma_1\int_{\Sigma_\lambda} \left(\frac{\hat{u}}{f_{x_0}(x)}-\frac{\hat{u}_\lambda}{f_{x_0}(x_\lambda)}\right) \xi_\lambda^+ \psi_\varepsilon^2\eta_R^2\,dx\,\,\\
		&+\int_{\Sigma_\lambda} (\hat{u}^{2^*-1}-\hat{u}_\lambda^{2^*-1}) \xi_\lambda^+ \psi_\varepsilon^2\eta_R^2\,dx\,\,\\
		&+ \nu \alpha\int_{\Sigma_\lambda} (\hat{u}^{\alpha-1}\hat{v}^\beta-\hat{u}^{\alpha-1}_\lambda \hat{v}_\lambda^\beta)\xi_\lambda^+ \psi_\varepsilon^2\eta_R^2\,dx\,\,\\
		=:&\,\mathcal{I}_1+\mathcal{I}_2+\mathcal{I}_3+\mathcal{I}_4+\mathcal{I}_5\,
		\end{split}
		\end{equation}
	and
		\begin{equation}\label{eq:diffv1}
		\begin{split}
		\int_{\Sigma_\lambda}|\nabla \zeta^+_\lambda|^2 \psi_\varepsilon^2\eta_R^2\,dx=&
		-2\int_{\Sigma_\lambda} \langle \nabla \zeta^+_\lambda, \nabla \psi_\varepsilon \rangle
		\zeta_\lambda^+ \psi_\varepsilon\eta_R^2\,dx
		-2\int_{\Sigma_\lambda} \langle \nabla \zeta^+_\lambda, \nabla \eta_R \rangle \zeta_\lambda^+ \eta_R \psi_\varepsilon^2\,dx\\
		&+\gamma_2\int_{\Sigma_\lambda} \left(\frac{\hat{v}}{f_{x_0}(x)}-\frac{\hat{v}_\lambda}{f_{x_0}(x_\lambda)}\right) \zeta_\lambda^+ \psi_\varepsilon^2\eta_R^2\,dx\,\,\\
		&+\int_{\Sigma_\lambda} (\hat{v}^{2^*-1}-\hat{v}_\lambda^{2^*-1}) \zeta_\lambda^+ \psi_\varepsilon^2\eta_R^2\,dx\,\,\\
		&+ \nu \beta\int_{\Sigma_\lambda} (\hat{u}^\alpha\hat{v}^{\beta-1}-\hat{u}^\alpha_\lambda \hat{v}_\lambda^{\beta-1})\zeta_\lambda^+ \psi_\varepsilon^2\eta_R^2\,dx\,\\
		=:&\,\mathcal{E}_1+\mathcal{E}_2+\mathcal{E}_3+\mathcal{E}_4+\mathcal{E}_5\,.
		\end{split}
		\end{equation}
		
		Exploiting also Young's inequality, recalling that
		$0\leq \xi_\lambda^+\leq \hat{u}$ and $0\leq \zeta_\lambda^+\leq \hat{v}$, we get that
		\begin{equation}\label{eq:I1}
		\begin{split}
		|\mathcal{I}_1|&\leq \frac{1}{4} \int_{\Sigma_\lambda}|\nabla \xi^+_\lambda|^2 \psi_\varepsilon^2\eta_R^2\,dx
		+4\int_{\Sigma_\lambda}| \nabla \psi_\varepsilon|^2(\xi_\lambda^+)^2\eta_R^2\,dx\\
		&\leq  \frac{1}{4} \int_{\Sigma_\lambda}|\nabla \xi^+_\lambda|^2 \psi_\varepsilon^2\eta_R^2\,dx
		+ 4 \varepsilon \|\hat{u}\|^2_{L^\infty(\Sigma_\lambda)} .\\
		\end{split}
		\end{equation}
	Similarly, we obtain
		\begin{equation}\label{eq:E1}
		\begin{split}
		|\mathcal{E}_1|&\leq  \frac{1}{4} \int_{\Sigma_\lambda}|\nabla \zeta^+_\lambda|^2 \psi_\varepsilon^2\eta_R^2\,dx
		+ 4 \varepsilon \|\hat{v}\|^2_{L^\infty(\Sigma_\lambda)} .\\
		\end{split}
		\end{equation}
		Furthermore, using H\"older's inequality with exponents $\left(\frac{2^*}{2}, \frac{n}{2}\right)$, one has that
		\begin{equation}\label{eq:I2}
		\begin{split}
		|\mathcal{I}_2|\leq& \frac{1}{4} \int_{\Sigma_\lambda}|\nabla \xi^+_\lambda|^2 \psi_\varepsilon^2\eta_R^2\,dx
		+4\int_{\Sigma_\lambda\cap(B_{2R}\setminus B_{R})}|\nabla \eta_R|^2(\xi_\lambda^+)^2 \psi_\varepsilon^2\,dx\\
		\leq&  \frac{1}{4} \int_{\Sigma_\lambda}|\nabla \xi^+_\lambda|^2 \psi_\varepsilon^2 \eta_R^2\,dx \\
		&+ 4\left(\int_{\Sigma_\lambda\cap(B_{2R}\setminus B_{R})}|\nabla
		\eta_R|^n\,dx\right)^{\frac{2}{n}}
		\left(\int_{\Sigma_\lambda\cap(B_{2R}\setminus B_{R})}\hat{u}^{2^*}\,dx\right)^{\frac{n-2}{n}}\\
		\leq&  \frac{1}{4} \int_{\Sigma_\lambda}|\nabla \xi^+_\lambda|^2 \psi_\varepsilon^2\eta_R^2\,dx\,+\,
		c(n) \left(\int_{\Sigma_\lambda\cap(B_{2R}\setminus B_{R})} \hat{u}^{2^*}\,dx\right)^{\frac{n-2}{n}}.
		\end{split}
		\end{equation}
where $c(n)$ is a positive constant depending only on the dimension $n$ and 
	\begin{center}
$\left(\int_{\Sigma_\lambda\cap(B_{2R}\setminus B_{R})} \hat{u}^{2^*}\,dx\right)^{\frac{n-2}{n}} \rightarrow 0 \qquad $ as $R \rightarrow +\infty$ 
	\end{center}
by the absolute continuity of the Lebesgue integral. Analogously, we deduce that
		\begin{equation}\label{eq:E2}
		\begin{split}
		|\mathcal{E}_2|\leq &  \frac{1}{4} \int_{\Sigma_\lambda}|\nabla \zeta^+_\lambda|^2 \psi_\varepsilon^2\eta_R^2\,dx\,+\,
		c(n) \left(\int_{\Sigma_\lambda\cap(B_{2R}\setminus B_{R})} \hat{v}^{2^*}\,dx\right)^{\frac{n-2}{n}}.
		\end{split}
		\end{equation}

		Let us now estimate $\mathcal{I}_3$ and $\mathcal{E}_3$. Here, we recall the monotonicity property of the function $f_{x_0}$ stated in Lemma \ref{lem:monotonicity}. Moreover, for $\lambda<0$ sufficiently large we have that
		$$\frac{1}{f_{x_0}(x_\lambda)} \geq \frac{1}{f_{x_0}(x)} .$$
		Therefore
		\begin{equation}
		\begin{split}
			|\mathcal{I}_3| &\leq \gamma_1\int_{\Sigma_\lambda} \frac{1}{f_{x_0}(x)} \left|\hat{u}- \hat{u}_\lambda\right| \xi_\lambda^+ \psi_\varepsilon^2\eta_R^2\,dx\,\,\\
			& \leq \gamma_1 C \int_{\Sigma_\lambda} \frac{(\xi_\lambda^+)^2}{|x|^4} \psi_\varepsilon^2\eta_R^2\,dx \leq \gamma_1 C C_\lambda \int_{\Sigma_\lambda} \frac{(\xi_\lambda^+)^2 \psi_\varepsilon^2 \eta_R^2}{|x|^2} \,dx,
		\end{split}
		\end{equation} 
		where 
		\begin{equation} \label{smallconstSUP}
			C_\lambda := \sup_{x \in \Sigma_{\lambda}} \frac{1}{|x|^2} \ \text{ and } \ C_\lambda \rightarrow 0 \ \text{ as } \ \lambda \rightarrow - \infty.
		\end{equation}
	 By Hardy's and Young's inequality we also deduce that
		\begin{equation} \label{eq:I3}
			\begin{split}
				|\mathcal{I}_3| \leq & \gamma_1 C_H^u C_\lambda \int_{\Sigma_\lambda} |\nabla (\xi_\lambda^+ \psi_\varepsilon \eta_R)|^2 \, dx\\				
				= & \gamma_1 C_H^u C_\lambda \int_{\Sigma_\lambda}  |\psi_\varepsilon \eta_R \nabla \xi_\lambda^+ + \xi_\lambda^+ \psi_\varepsilon \nabla \eta_R + \xi_\lambda^+ \eta_R \nabla \psi_\varepsilon|^2\,dx\\
				\leq & 3 \gamma_1 C_H^u C_\lambda \int_{\Sigma_\lambda}  | \nabla \xi_\lambda^+|^2 \psi_\varepsilon^2 \eta_R^2 \, dx + 3 \gamma_1 C_H^u C_\lambda \int_{\Sigma_\lambda \cap(B_{2R}\setminus B_{R})} |\nabla \eta_R|^2 (\xi_\lambda^+)^2 \psi_\varepsilon^2 \, dx \\
				&+ 3 \gamma_1 C_H^u C_\lambda \int_{\Sigma_\lambda}  |\nabla \psi_\varepsilon|^2 (\xi_\lambda^+)^2 \eta_R^2 \,dx\\
				\leq & 3 \gamma_1 C_H^u C_\lambda \int_{\Sigma_\lambda}  | \nabla \xi_\lambda^+|^2 \psi_\varepsilon^2 \eta_R^2 \, dx + 3 \gamma_1 C_H^u C_\lambda c(n) \left(\int_{\Sigma_\lambda\cap(B_{2R}\setminus B_{R})} \hat{u}^{2^*}\,dx\right)^{\frac{n-2}{n}} \\
				& + 3 \gamma_1 C_H^u C_\lambda \varepsilon \|\hat{u}\|^2_{L^\infty(\Sigma_\lambda)} .
			\end{split}
		\end{equation} 
	Employing the argument to estimate $\mathcal{I}_3$, we get
		\begin{equation} \label{eq:E3}
			\begin{split}
				|\mathcal{E}_3| \leq & 3 \gamma_2 C_H^v C_\lambda \int_{\Sigma_\lambda}  | \nabla \zeta_\lambda^+|^2 \psi_\varepsilon^2 \eta_R^2 \, dx + 3 \gamma_2 C_H^v C_\lambda c(n) \left(\int_{\Sigma_\lambda\cap(B_{2R}\setminus B_{R})} \hat{v}^{2^*}\,dx\right)^{\frac{n-2}{n}} \\
				&+ 3 \gamma_2 C_H^v C_\lambda \varepsilon \|\hat{v}\|^2_{L^\infty(\Sigma_\lambda)}.
			\end{split}
		\end{equation} 
		Since $\hat{u}(x), \hat{u}_\lambda(x), \hat{v}(x), \hat{v}_\lambda(x)>0$, by the
		convexity of $ t \mapsto t^{2^*-1},$ for $ t >0$, we obtain
		$$\hat{u}^{2^*-1}(x)-\hat{u}_\lambda^{2^*-1}(x) \le  \frac{n+2}{n-2}
		\hat{u}^{2^*-2}(x) (\hat{u}(x) - \hat{u}_\lambda(x))$$ and $$\hat{v}^{2^*-1}(x)-\hat{v}_\lambda^{2^*-1}(x) \le  \frac{n+2}{n-2}
		\hat{v}^{2^*-2}(x) (\hat{v}(x) - \hat{v}_\lambda(x)),$$ for every $x \in
		\Sigma_{\lambda}$. Thus, 
		$$(\hat{u}^{2^*-1}-\hat{u}_\lambda^{2^*-1})\xi_\lambda^+ \le  \frac{n+2}{n-2}
		\hat{u}^{2^*-2}(\hat{u}-\hat{u}_\lambda) \xi_\lambda^+ \le \frac{n+2}{n-2}
		\hat{u}^{2^*-2}(\xi_\lambda^+)^2$$ and $$(\hat{v}^{2^*-1}-\hat{v}_\lambda^{2^*-1})\zeta_\lambda^+ \le  \frac{n+2}{n-2}
		\hat{v}^{2^*-2}(\hat{v}-\hat{v}_\lambda) \zeta_\lambda^+ \le \frac{n+2}{n-2}
		\hat{v}^{2^*-2}(\zeta_\lambda^+)^2.$$ 
		Therefore, using H\"older's inequality with exponents $\left(\frac{2^*}{2}, \frac{n}{2}\right)$, we have
		\begin{equation}\label{eq:I4}
		\begin{split}
		|\mathcal{I}_4| & \leq \frac{n+2}{n-2} \int_{\Sigma_\lambda} \hat{u}^{2^*-2} (\xi_\lambda^+ \psi_\varepsilon \eta_R)^{2} \, dx \\
		&\leq 		\frac{n+2}{n-2} \left(\int_{\Sigma_\lambda}
		\hat{u}^{2^*} \,dx \right)^\frac{2}{n} \left(\int_{\Sigma_\lambda} (\xi_\lambda^+ \psi_\varepsilon \eta_R)^{2^*}\,dx \right)^\frac{n-2}{n} \\
		&\leq 		\frac{n+2}{n-2} \mathcal{C}_S^u \left(\int_{\Sigma_\lambda}
		\hat{u}^{2^*} \,dx \right)^\frac{2}{n} \int_{\Sigma_\lambda} |\nabla (\xi_\lambda^+ \psi_\varepsilon \eta_R)|^{2}\,dx \\
		&\leq 		3 \frac{n+2}{n-2} \mathcal{C}_S^u \left(\int_{\Sigma_\lambda}
		\hat{u}^{2^*} \,dx \right)^\frac{2}{n} \left[\int_{\Sigma_\lambda} |\nabla \eta_R|^2 (\xi_\lambda^+)^2 \psi_\varepsilon^2 \,dx +\int_{\Sigma_\lambda} |\nabla \psi_\varepsilon |^2 (\xi_\lambda^+)^2 \eta_R^{2}\,dx \right]\\
		& \quad + 3 \frac{n+2}{n-2} \mathcal{C}_S^u  \left(\int_{\Sigma_\lambda}\hat{u}^{2^*} \,dx \right)^\frac{2}{n} \int_{\Sigma_\lambda} |\nabla \xi_\lambda^+|^2 \psi_\varepsilon^2 \eta_R^{2}\,dx,
		\end{split}
		\end{equation}
	where in the last two steps we applied Sobolev and Young's inequalities respectively. Arguing in the same way, we deduce 
		\begin{equation}\label{eq:E4}
		\begin{split}
		|\mathcal{E}_4|& \leq 	3	\frac{n+2}{n-2} \mathcal{C}_S^v \left(\int_{\Sigma_\lambda}
		\hat{v}^{2^*} \,dx \right)^\frac{2}{n} \left[\int_{\Sigma_\lambda} |\nabla \eta_R|^2 (\zeta_\lambda^+)^2 \psi_\varepsilon^2 \,dx +\int_{\Sigma_\lambda} |\nabla \psi_\varepsilon |^2 (\zeta_\lambda^+)^2 \eta_R^{2}\,dx \right]\\
		& \quad + 3 \frac{n+2}{n-2} \mathcal{C}_S^v  \left(\int_{\Sigma_\lambda}\hat{v}^{2^*} \,dx \right)^\frac{2}{n} \int_{\Sigma_\lambda} |\nabla \zeta_\lambda^+|^2 \psi_\varepsilon^2 \eta_R^{2}\,dx.
		\end{split}
		\end{equation}
	The evaluation of $\mathcal{I}_5$ and $\mathcal{E}_5$ is a delicate issue within this argument. Note that, in particular, we may have that either $\alpha<2$, $\beta<2$ or both. In all this cases we have to face a nonlinear term which is not Lipschitz continuous at zero. We shall exploit the following
	\begin{rem}\label{fsdfd}
Let us consider $h(t)=t^{s}$ with $s>0$  and $b\in[\theta_1\,,\, \theta_2]$ with $\theta_i>0$.
Then, for $0<a\leq b$
\[
\frac{h(b)-h(a)}{b-a}\leq C(s,\theta_1,\theta_2)\,.
\]
To prove this fact it is sufficient to exploit the Mean Value Theorem for $a\in [b/2,b]$ and the
Weierstrass Theorem  for $a\in (0,b/2]$\,.
	\end{rem}
	Having in mind this argument, one obtains that
	
		\begin{equation}\label{eq:I_50}
		\begin{split}
			\mathcal{I}_5 &= \nu \alpha\int_{\Sigma_\lambda} (\hat{u}^{\alpha-1}\hat{v}^\beta-\hat{u}^{\alpha-1}_\lambda \hat{v}^\beta)\xi_\lambda^+ \psi_\varepsilon^2\eta_R^2\,dx + \nu \alpha\int_{\Sigma_\lambda} (\hat{u}^{\alpha-1}_\lambda \hat{v}^\beta-\hat{u}^{\alpha-1}_\lambda \hat{v}_\lambda^\beta)\xi_\lambda^+ \psi_\varepsilon^2\eta_R^2\,dx\\
			& \leq \nu \alpha\int_{\Sigma_\lambda} (\hat{u}^{\alpha-1} - \hat{u}^{\alpha-1}_\lambda) \hat{v}^\beta \xi_\lambda^+ \psi_\varepsilon^2\eta_R^2\,dx + \nu \alpha\int_{\Sigma_\lambda  \cap \{\hat v > \hat v_\lambda\}} \hat{u}^{\alpha-1} (\hat{v}^\beta- \hat{v}_\lambda^\beta)\xi_\lambda^+ \psi_\varepsilon^2\eta_R^2\,dx,
		\end{split}
	\end{equation}
	where in the last inequality we used the cooperativity of our system and the fact that we are working in $\Sigma_\lambda \cap \text{supp}[(u-u_\lambda)^+]$.
	 Making use of Lemma \ref{lem:asBehaviour}  we deduce that
\begin{equation}\label{eq:I_5}
	\begin{split}
		\mathcal{I}_5 &   \leq \nu \alpha\int_{\Sigma_\lambda} |\hat{u}^{\alpha-1} - \hat{u}^{\alpha-1}_\lambda| \hat{v}^\beta \xi_\lambda^+ \psi_\varepsilon^2\eta_R^2\,dx + \nu \alpha \int_{\Sigma_\lambda  \cap \{\hat v > \hat v_\lambda\}} \hat{u}^{\alpha-1} |\hat{v}^\beta- \hat{v}_\lambda^\beta|\xi_\lambda^+ \psi_\varepsilon^2\eta_R^2\,dx\\	
		& \leq  \nu \alpha \mathcal{C} \int_{\Sigma_\lambda} |x|^{-\beta(n-2)}|x|^{-(\alpha-1)(n-2)}|(|x|^{n-2}\hat{u})^{\alpha-1} - (|x|^{n-2}\hat{u}_\lambda)^{\alpha-1}| \xi_\lambda^+ \psi_\varepsilon^2\eta_R^2\,dx\\
		& \quad + \nu \alpha \mathcal{C} \int_{\Sigma_\lambda \cap \{\hat v > \hat v_\lambda\}} |x|^{-(\alpha-1)(n-2)}|x|^{-(\beta-1)(n-2)}
		\xi_\lambda^+ \zeta_\lambda^+ \psi_\varepsilon^2\eta_R^2\,dx \\
		& \leq  \nu \alpha \mathcal{C} \int_{\Sigma_\lambda} |x|^{-(2^*-1)(n-2)} \frac{|(|x|^{n-2}\hat{u})^{\alpha-1} - (|x|^{n-2}\hat{u}_\lambda)^{\alpha-1}|}{|x|^{n-2} \hat u - |x|^{n-2} \hat u_\lambda} |x|^{n-2}(\xi_\lambda^+)^2 \psi_\varepsilon^2\eta_R^2\,dx\\
		& \quad + \nu \alpha \mathcal{C} \int_{\Sigma_\lambda} |x|^{-(2^*-2)(n-2)}	\xi_\lambda^+ \zeta_\lambda^+ \psi_\varepsilon^2\eta_R^2\,dx. \\
	\end{split}
\end{equation}

By Remark \ref{fsdfd}, we get
$$\frac{(|x|^{n-2}\hat{u})^{\alpha-1} - (|x|^{n-2}\hat{u}_\lambda)^{\alpha-1}}{|x|^{n-2} \hat u - |x|^{n-2} \hat u_\lambda} \leq \overline{C},$$
whereas $c_u \leq |x|^{n-2}\hat u \leq C_u$ and $c_v \leq |x|^{n-2}\hat v \leq C_v$ by Lemma~\ref{lem:asBehaviour}. Then \eqref{eq:I_5} rewrites as 
\begin{equation}\label{eq:I_51}
	\begin{split}		
		|\mathcal{I}_5| &  \leq \nu \alpha \mathcal{C} \int_{\Sigma_\lambda} \hat u^{2^*-2} (\xi_\lambda^+)^2 \psi_\varepsilon^2\eta_R^2\,dx\, + \nu \alpha \mathcal{C} \int_{\Sigma_\lambda} \hat u^{\frac{2^*-2}{2}}	\xi_\lambda^+  \psi_\varepsilon \eta_R \hat v^{\frac{2^*-2}{2}} \zeta_\lambda^+ \psi_\varepsilon \eta_R \,dx\\
		& \leq \mathcal{C}_1 \int_{\Sigma_\lambda} \hat u^{2^*-2}(\xi_\lambda^+)^2 \psi_\varepsilon^2\eta_R^2\,dx\, + \mathcal{C}_2 \int_{\Sigma_\lambda} \hat v^{2^*-2}	(\zeta_\lambda^+)^2 \psi_\varepsilon^2\eta_R^2\,dx,
	\end{split}
\end{equation}
where the constant $\mathcal{C}$ has been relabelled and we have applied Young's inequality.

Similarly, we can obtain an analogous estimate for $\mathcal{E}_5$, i.e.
\begin{equation}\label{eq:E_5}
	\begin{split}
		|\mathcal{E}_5|  \leq \tilde{\mathcal{C}}_1 \int_{\Sigma_\lambda} \hat u^{2^*-2}(\xi_\lambda^+)^2 \psi_\varepsilon^2\eta_R^2\,dx\, + \tilde{\mathcal{C}}_2 \int_{\Sigma_\lambda} \hat v^{2^*-2}	(\zeta_\lambda^+)^2 \psi_\varepsilon^2\eta_R^2\,dx.
    \end{split}
\end{equation}
As we argued in \eqref{eq:I4} for $\mathcal{I}_4$ in, we get that
	\begin{equation}\label{eq:I_5bis}
	\begin{split}
		|\mathcal{I}_5| & \leq 		3 \frac{n+2}{n-2} \mathcal{C}_S^u \mathcal{C}_1 \left(\int_{\Sigma_\lambda}
		\hat{u}^{2^*} \,dx \right)^\frac{2}{n} \left[\int_{\Sigma_\lambda} |\nabla \eta_R|^2 (\xi_\lambda^+)^2 \psi_\varepsilon^2 \,dx +\int_{\Sigma_\lambda} |\nabla \psi_\varepsilon |^2 (\xi_\lambda^+)^2 \eta_R^{2}\,dx \right]  \\
		& \quad  	+	3\frac{n+2}{n-2} \mathcal{C}_S^v \mathcal{C}_2\left(\int_{\Sigma_\lambda}
		\hat{v}^{2^*} \,dx \right)^\frac{2}{n} \left[\int_{\Sigma_\lambda} |\nabla \eta_R|^2 (\zeta_\lambda^+)^2 \psi_\varepsilon^2 \,dx +\int_{\Sigma_\lambda} |\nabla \psi_\varepsilon |^2 (\zeta_\lambda^+)^2 \eta_R^{2}\,dx \right]  \\
		& \quad + 3 \frac{n+2}{n-2} \mathcal{C}_S^u \mathcal{C}_1 \left(\int_{\Sigma_\lambda}\hat{u}^{2^*} \,dx \right)^\frac{2}{n} \int_{\Sigma_\lambda} |\nabla \xi_\lambda^+|^2 \psi_\varepsilon^2 \eta_R^{2}\,dx\\
		& \quad + 3 \frac{n+2}{n-2} \mathcal{C}_S^v \mathcal{C}_2 \left(\int_{\Sigma_\lambda}\hat{v}^{2^*} \,dx \right)^\frac{2}{n} \int_{\Sigma_\lambda} |\nabla \zeta_\lambda^+|^2 \psi_\varepsilon^2 \eta_R^{2}\,dx,
	\end{split}
	\end{equation}
Analogously, we deduce that 
	\begin{equation}\label{eq:E_5bis}
	\begin{split}
		|\mathcal{E}_5|  &	\leq 		3 \frac{n+2}{n-2} \mathcal{C}_S^u \tilde{\mathcal{C}}_1 \left(\int_{\Sigma_\lambda}
		\hat{u}^{2^*} \,dx \right)^\frac{2}{n} \left[\int_{\Sigma_\lambda} |\nabla \eta_R|^2 (\xi_\lambda^+)^2 \psi_\varepsilon^2 \,dx +\int_{\Sigma_\lambda} |\nabla \psi_\varepsilon |^2 (\xi_\lambda^+)^2 \eta_R^{2}\,dx \right]\\
		& \quad  	+	3\frac{n+2}{n-2} \mathcal{C}_S^v \tilde{\mathcal{C}}_2\left(\int_{\Sigma_\lambda}
		\hat{v}^{2^*} \,dx \right)^\frac{2}{n} \left[\int_{\Sigma_\lambda} |\nabla \eta_R|^2 (\zeta_\lambda^+)^2 \psi_\varepsilon^2 \,dx +\int_{\Sigma_\lambda} |\nabla \psi_\varepsilon |^2 (\zeta_\lambda^+)^2 \eta_R^{2}\,dx \right]\\
		& \quad + 3 \frac{n+2}{n-2} \mathcal{C}_S^u \tilde{\mathcal{C}}_1 \left(\int_{\Sigma_\lambda}\hat{u}^{2^*} \,dx \right)^\frac{2}{n} \int_{\Sigma_\lambda} |\nabla \xi_\lambda^+|^2 \psi_\varepsilon^2 \eta_R^{2}\,dx\\
		& \quad + 3 \frac{n+2}{n-2} \mathcal{C}_S^v \tilde{\mathcal{C}}_2 \left(\int_{\Sigma_\lambda}\hat{v}^{2^*} \,dx \right)^\frac{2}{n} \int_{\Sigma_\lambda} |\nabla \zeta_\lambda^+|^2 \psi_\varepsilon^2 \eta_R^{2}\,dx.
	\end{split}
\end{equation}

Collecting all the previous estimates for $\mathcal{I}_k, \mathcal{E}_k$ with $k=1,2,3,4,5$, we deduce that 		
\begin{equation}\label{eq:diffu2}
	\begin{split}
		\int_{\Sigma_\lambda}|\nabla \xi^+_\lambda|^2 \psi_\varepsilon^2\eta_R^2\,dx \leq&
		C_{1,u}\int_{\Sigma_\lambda}|\nabla \xi^+_\lambda|^2 \psi_\varepsilon^2\eta_R^2\,dx + C_{1,v} \int_{\Sigma_\lambda} |\nabla \zeta_\lambda^+|^2 \psi_\varepsilon^2 \eta_R^{2}\,dx,\\
		& + \varepsilon C_{2,u} \|\hat{u}\|^2_{L^\infty(\Sigma_\lambda)} + 	\varepsilon C_{2,v} \|\hat{v}\|^2_{L^\infty(\Sigma_\lambda)}\\
		& + C_{3,u} \left(\int_{\Sigma_\lambda\cap(B_{2R}\setminus B_{R})} \hat{u}^{2^*}\,dx\right)^{\frac{n-2}{n}} +	C_{3,v} \left(\int_{\Sigma_\lambda\cap(B_{2R}\setminus B_{R})} \hat{v}^{2^*}\,dx\right)^{\frac{n-2}{n}},
	\end{split}
\end{equation}
where
\begin{itemize}
	\item[] $\displaystyle C_{1,u} := \frac{1}{2} +3\gamma_1 C_H^u C_\lambda + 3 \frac{n+2}{n-2} \mathcal{C}_S^u (1+\mathcal{C}_1)\|\hat u\|_{L^{2^*}(\Sigma_\lambda)}^{\frac{4}{n-2}},$
	
	\item[] $C_{1,v} := \displaystyle 3 \frac{n+2}{n-2} \mathcal{C}_S^v \mathcal{C}_2\|\hat v\|_{L^{2^*} (\Sigma_\lambda)}^{\frac{4}{n-2}},$
		
	\item[] $\displaystyle C_{2,u} := 4 + 3 \gamma_1 C_H^u C_\lambda + 3 \frac{n+2}{n-2} \mathcal{C}_S^u (1+\mathcal{C}_1)\|\hat u\|_{L^{2^*}(\Sigma_\lambda)}^{\frac{4}{n-2}},$
	
	\item[] $\displaystyle C_{2,v} := 3 \frac{n+2}{n-2}  \mathcal{C}_S^v \mathcal{C}_2\|\hat v\|_{L^{2^*}(\Sigma_\lambda)}^{\frac{4}{n-2}},$
	
	\item[] $\displaystyle C_{3,u} := c(n) + 3 \gamma_1 C_H^u C_\lambda c(n)  + 3 \frac{n+2}{n-2} \mathcal{C}_S^u (1+\mathcal{C}_1)\|\hat u\|_{L^{2^*}(\Sigma_\lambda)}^{\frac{4}{n-2}},$
	
	\item[] $\displaystyle C_{3,v} := 3 \frac{n+2}{n-2}  \mathcal{C}_S^v \mathcal{C}_2\|\hat v\|_{L^{2^*}(\Sigma_\lambda)}^{\frac{4}{n-2}}.$
\end{itemize}
We obtain an analogous estimate for $\zeta_\lambda^+$, i.e.
\begin{equation}\label{eq:diffv2}
	\begin{split}
		\int_{\Sigma_\lambda}|\nabla \zeta^+_\lambda|^2 \psi_\varepsilon^2\eta_R^2\,dx \leq & \tilde C_{1,u}\int_{\Sigma_\lambda}|\nabla \xi^+_\lambda|^2 \psi_\varepsilon^2\eta_R^2\,dx + \tilde C_{1,v} \int_{\Sigma_\lambda} |\nabla \zeta_\lambda^+|^2 \psi_\varepsilon^2 \eta_R^{2}\,dx,\\
		& + \varepsilon \tilde C_{2,u} \|\hat{u}\|^2_{L^\infty(\Sigma_\lambda)} + 	\varepsilon \tilde C_{2,v} \|\hat{v}\|^2_{L^\infty(\Sigma_\lambda)}\\
		& + \tilde C_{3,u} \left(\int_{\Sigma_\lambda\cap(B_{2R}\setminus B_{R})} \hat{u}^{2^*}\,dx\right)^{\frac{n-2}{n}} + \tilde 	C_{3,v} \left(\int_{\Sigma_\lambda\cap(B_{2R}\setminus B_{R})} \hat{v}^{2^*}\,dx\right)^{\frac{n-2}{n}},
	\end{split}
\end{equation}
where
\begin{itemize}
	\item[] $\displaystyle \tilde C_{1,u} := 3 \frac{n+2}{n-2} \mathcal{C}_S^u \tilde{\mathcal{C}}_1\|\hat u\|_{L^{2^*} (\Sigma_\lambda)}^{\frac{4}{n-2}},$

	\item[] $\displaystyle \tilde C_{1,v} :=  \frac{1}{2} +3\gamma_1 C_H^v C_\lambda + 3 \frac{n+2}{n-2} \mathcal{C}_S^v (1+\tilde{\mathcal{C}}_2)\|\hat v\|_{L^{2^*}(\Sigma_\lambda)}^{\frac{4}{n-2}},$

	\item[] $\displaystyle \tilde C_{2,u} := 3 \frac{n+2}{n-2}  \mathcal{C}_S^u \tilde{\mathcal{C}}_1\|\hat u\|_{L^{2^*}(\Sigma_\lambda)}^{\frac{4}{n-2}},$

	\item[] $\displaystyle \tilde C_{2,v} := 4 + 3 \gamma_1 C_H^v C_\lambda + 3 \frac{n+2}{n-2} \mathcal{C}_S^v (1+\tilde{\mathcal{C}}_2)\|\hat v\|_{L^{2^*}(\Sigma_\lambda)}^{\frac{4}{n-2}},$

	\item[] $\displaystyle \tilde C_{3,u} :=  3 \frac{n+2}{n-2} \mathcal{C}_S^u \tilde{\mathcal{C}}_1\|\hat u\|_{L^{2^*}(\Sigma_\lambda)}^{\frac{4}{n-2}},$

	\item[] $\displaystyle \tilde C_{3,v} := c(n) + 3 \gamma_1 C_H^u C_\lambda c(n)  + 3 \frac{n+2}{n-2}  \mathcal{C}_S^v \tilde{\mathcal{C}}_2\|\hat v\|_{L^{2^*}(\Sigma_\lambda)}^{\frac{4}{n-2}}.$
\end{itemize}
Summing both the contributions \eqref{eq:diffu2} and \eqref{eq:diffv2} we get
\begin{equation}\label{eq:summfinal}
	\begin{split}
		\int_{\Sigma_\lambda}|\nabla \xi^+_\lambda|^2 \psi_\varepsilon^2\eta_R^2\,dx & +  \int_{\Sigma_\lambda}|\nabla \zeta^+_\lambda|^2 \psi_\varepsilon^2\eta_R^2\,dx \\		
		\leq&	(C_{1,u} + \tilde C_{1,u})\int_{\Sigma_\lambda}|\nabla \xi^+_\lambda|^2 \psi_\varepsilon^2\eta_R^2\,dx + (C_{1,v} + \tilde C_{1,v}) \int_{\Sigma_\lambda} |\nabla \zeta_\lambda^+|^2 \psi_\varepsilon^2 \eta_R^{2}\,dx\\
		& + \varepsilon (C_{2,u} + \tilde C_{2,u}) \|\hat{u}\|^2_{L^\infty(\Sigma_\lambda)} + 	\varepsilon (C_{2,v} + C_{2,v}) \|\hat{v}\|^2_{L^\infty(\Sigma_\lambda)}\\
		& + (C_{3,u} + \tilde C_{3,u}) \left(\int_{\Sigma_\lambda\cap(B_{2R}\setminus B_{R})} \hat{u}^{2^*}\,dx\right)^{\frac{n-2}{n}} \\
		& +	(C_{3,v} + \tilde C_{3,v}) \left(\int_{\Sigma_\lambda\cap(B_{2R}\setminus B_{R})} \hat{v}^{2^*}\,dx\right)^{\frac{n-2}{n}}.
	\end{split}
\end{equation}

Now, recalling \eqref{smallconstSUP} and using the absolute continuity of the Lebesgue integral, we fix $M> 0$ sufficiently large such that 
$$ C_{1,u} + \tilde C_{1,u} < 1 \quad \text{and} \quad C_{1,v} + \tilde C_{1,v} < 1$$
for each $\lambda < - M$. Finally, passing to the limit for $R \rightarrow + \infty$ and $\varepsilon \rightarrow 0^+$, we obtain that
\begin{equation}
	\int_{\Sigma_\lambda}|\nabla \xi^+_\lambda|^2  \, dx + \int_{\Sigma_\lambda}|\nabla \zeta^+_\lambda|^2  \, dx \leq 0,
\end{equation}
for each $\lambda < - M$, which implies
\begin{equation}
	\int_{\Sigma_\lambda}|\nabla \xi^+_\lambda|^2  \, dx \leq 0 \quad \text{and} \quad \int_{\Sigma_\lambda}|\nabla \zeta^+_\lambda|^2  \, dx \leq 0,
\end{equation}
for each $\lambda < - M$. The last inequalities immediately imply the thesis of the first step.

\

	\noindent To proceed further we
	define
	\begin{equation}\nonumber
	\Lambda=\{\lambda<0 : \hat{u}\leq \hat{u}_{t}\,\,\,\text{and}\,\,\, \hat{v}\leq \hat{v}_{t}\,\,\,\text{in} \,\,\, \Sigma_t\setminus R_t(\{0, x_0\})\,\,\,\text{for all $t\in(-\infty,\lambda]$}\}
	\end{equation}
	and
	\begin{equation}\nonumber
	\bar \lambda=\sup\,\Lambda.
	\end{equation}

	\noindent \emph{Step 2: $\bar \lambda=0$.} 
	We argue by
	contradiction and suppose that $\bar \lambda<0$. By continuity we know
	that $\hat{u}\leq \hat{u}_{\bar \lambda}$ and $\hat{v}\leq \hat{v}_{\bar \lambda}$ in $\Sigma_{\bar \lambda}\setminus R_{\bar \lambda}(\{0, x_0\})$. By the strong comparison principle we deduce that there hold two possibilities:
	\begin{itemize}
		\item[(i)] $\hat u < \hat u_{\bar \lambda}$ and $\hat v < \hat v_{\bar \lambda}$ in $\Sigma_{\bar \lambda} \setminus R_{\bar \lambda}(\{0,x_0\})$;
			
		\item[(ii)]  $\hat u \equiv \hat u_{\bar \lambda}$ and $\hat v \equiv \hat v_{\bar \lambda}$ in $\Sigma_{\bar \lambda} \setminus R_{\bar \lambda}(\{0,x_0\})$.
	\end{itemize}
	The case (ii) is not possible ,because, since $\hat u $ has a sign at $x_0$, it holds that $\hat u < \hat u_{\bar \lambda}$ and $\hat v < \hat v_{\bar \lambda}$ in $\Sigma_{\bar \lambda} \setminus R_{\bar \lambda}(\{0, x_0\})$.
	
	Now, we push further the hyperplane $T_{\bar \lambda}$ and consider the hyperplane 
	$T_{\bar \lambda+\varepsilon}$ for some $\varepsilon>0$. 
We claim that,  for $\varepsilon> 0$ small, we have that
	\begin{equation}\label{eq:compactsupptest}
		\text{supp}[(\hat u - \hat u_{\bar \lambda + \varepsilon})^+], \ \text{supp}[(\hat v - \hat v_{\bar \lambda + \varepsilon})^+] \subset (\Sigma_{\bar \lambda + \varepsilon} \setminus B_{\tilde R}(0)) \cup R_{\bar \lambda +\varepsilon}(B_\delta(0)  \cup  B_\delta(x_0)),
	\end{equation}
	where $\tilde R>0$ is given by Lemma \ref{lem:asBehaviour} and $\delta>0$ is small. For the reader convenience we set
	$$\Xi_{\bar \lambda + \varepsilon} := (\Sigma_{\bar \lambda + \varepsilon} \setminus B_{\tilde R}(0)) \cup R_{\bar \lambda +\varepsilon}(B_\delta(0)  \cup  B_\delta(x_0)).$$
	
	Arguing by contradiction, let us assume that \eqref{eq:compactsupptest} does not hold. Hence, if we define $$\mathcal{G}_\varepsilon:= B_{\tilde R}(0) \cap (\Sigma_{\bar \lambda + \varepsilon} \setminus \{R_{\bar \lambda +\varepsilon}(B_\delta(0)  \cup  B_\delta(x_0))\}),$$
	we deduce that 
\begin{equation*}
	\begin{split}
		\\ \;  \exists \  P_m \in \mathcal{G}_{ \tau_m},  \text{ with }  \{\tau_m\}_{m \in \N} \ (\tau_m \rightarrow 0) \text{ such that } \\
		\hat{u} (P_m) > \hat u_{\bar \lambda + \tau_m} (P_m) \;  \text{  or } \; \hat{v} (P_m) > \hat v_{\bar \lambda + \tau_m}(P_m).	
	\end{split}
\end{equation*}
Without loss of generality, we assume that $\hat{u} (P_m) > \hat u_{\bar \lambda + \tau_m} (P_m)$. Up to subsequence, as $m \rightarrow + \infty$ we have that
$$P_m \longrightarrow \bar P \in \overline{	\mathcal{G}_{0}}.$$
Finally we have that or $\bar P \in T_{\bar \lambda}$ or $\bar P \notin  T_{\bar \lambda}$.

\begin{itemize}
	\item If $\bar P \notin  T_{\bar \lambda}$, then we have a contradiction with the fact that $\hat u(\bar P)< \hat u_{\bar \lambda}(\bar P) $ and $\bar P \not \in R_{\bar \lambda }(B_\delta(0)  \cup  B_\delta(x_0))$.
	
	\item If $\bar P \in T_{\bar \lambda}$, then we have $\displaystyle \frac{\partial \hat u}{\partial x_1} ( \bar P) \leq 0$. But, by the Hopf boundary lemma, we know that $\displaystyle \frac{\partial \hat u}{\partial x_1} ( \bar P) > 0$ providing a contradiction. Note that here it is needed to run over the argument in 
	 \cite{GNN} that works thanks to the cooperativity condition.
\end{itemize}
Now, we know that \eqref{eq:compactsupptest} holds true  and we argue as in \textit{Step 1} choosing
	$$\Phi\,:= \begin{cases}
	\, \xi_{\bar \lambda + \varepsilon}^+\psi_\varepsilon^2\eta_R^2 \, & \text{in}\quad\Sigma_{\bar \lambda + \varepsilon}, \\
	0 &  \text{in}\quad \R^n \setminus  \Sigma_{\bar \lambda + \varepsilon}
\end{cases} \quad \text{and} \quad
\Psi\,:= \begin{cases}
	\, \zeta_{\bar \lambda + \varepsilon}^+\psi_\varepsilon^2\eta_R^2 \, & \text{in}\quad\Sigma_{\bar \lambda + \varepsilon}, \\
	0 &  \text{in}\quad \R^n \setminus  \Sigma_{\bar \lambda + \varepsilon},
\end{cases}$$
as test functions respectively in \eqref{eq:weakKelv3} and in \eqref{eq:weakKelv4} so that, subtracting we get
		\begin{equation}\label{eq:diffu1bis}
	\begin{split}
		\int_{\Xi_{\bar \lambda + \varepsilon} }|\nabla \xi^+_{\bar \lambda + \varepsilon}|^2 \psi_\varepsilon^2\eta_R^2\,dx=&
		-2\int_{\Xi_{\bar \lambda + \varepsilon}}\langle \nabla \xi^+_{\bar \lambda + \varepsilon}, \nabla \psi_\varepsilon \rangle
		\xi_{\bar \lambda + \varepsilon}^+ \psi_\varepsilon\eta_R^2\,dx\\
		&-2\int_{\Xi_{\bar \lambda + \varepsilon}} \langle \nabla \xi^+_{\bar \lambda + \varepsilon}, \nabla \eta_R \rangle \xi_{\bar \lambda + \varepsilon}^+ \eta_R \psi_\varepsilon^2\,dx\\
		&+\gamma_1\int_{\Xi_{\bar \lambda + \varepsilon}} \left(\frac{\hat{u}}{f_{x_0}(x)}-\frac{\hat{u}_{\bar \lambda + \varepsilon}}{f_{x_0}(x_{\bar \lambda + \varepsilon})}\right) \xi_{\bar \lambda + \varepsilon}^+ \psi_\varepsilon^2\eta_R^2\,dx\,\,\\
		&+\int_{\Xi_{\bar \lambda + \varepsilon}} (\hat{u}^{2^*-1}-\hat{u}_{\bar \lambda + \varepsilon}^{2^*-1}) \xi_{\bar \lambda + \varepsilon}^+ \psi_\varepsilon^2\eta_R^2\,dx\,\,\\
		&+ \nu \alpha\int_{\Xi_{\bar \lambda + \varepsilon}} (\hat{u}^{\alpha-1}\hat{v}^\beta-\hat{u}^{\alpha-1}_{\bar \lambda + \varepsilon} \hat{v}_{\bar \lambda + \varepsilon}^\beta) \xi_{\bar \lambda + \varepsilon}^+ \psi_\varepsilon^2\eta_R^2\,dx\,.
	\end{split}
\end{equation}
\begin{equation}\label{eq:diffv1bis}
	\begin{split}
	\int_{\Xi_{\bar \lambda + \varepsilon} }|\nabla \zeta^+_{\bar \lambda + \varepsilon}|^2 \psi_\varepsilon^2\eta_R^2\,dx=&
	-2\int_{\Xi_{\bar \lambda + \varepsilon}}\langle \nabla \zeta^+_{\bar \lambda + \varepsilon}, \nabla \psi_\varepsilon \rangle
	\zeta_{\bar \lambda + \varepsilon}^+ \psi_\varepsilon\eta_R^2\,dx\\
	&-2\int_{\Xi_{\bar \lambda + \varepsilon}} \langle \nabla \zeta^+_{\bar \lambda + \varepsilon}, \nabla \eta_R \rangle \zeta_{\bar \lambda + \varepsilon}^+ \eta_R \psi_\varepsilon^2\,dx\\
	&+\gamma_1\int_{\Xi_{\bar \lambda + \varepsilon}} \left(\frac{\hat{v}}{f_{x_0}(x)}-\frac{\hat{v}_{\bar \lambda + \varepsilon}}{f_{x_0}(x_{\bar \lambda + \varepsilon})}\right) \zeta_{\bar \lambda + \varepsilon}^+ \psi_\varepsilon^2\eta_R^2\,dx\,\,\\
	&+\int_{\Xi_{\bar \lambda + \varepsilon}} (\hat{v}^{2^*-1}-\hat{v}_{\bar \lambda + \varepsilon}^{2^*-1}) \zeta_{\bar \lambda + \varepsilon}^+ \psi_\varepsilon^2\eta_R^2\,dx\,\,\\
	&+ \nu \beta \int_{\Xi_{\bar \lambda + \varepsilon}} (\hat{u}^{\alpha}\hat{v}^{\beta-1}-\hat{u}^{\alpha}_{\bar \lambda + \varepsilon} \hat{v}_{\bar \lambda + \varepsilon}^{\beta-1})\zeta_{\bar \lambda + \varepsilon}^+ \psi_\varepsilon^2\eta_R^2\,dx\,.
\end{split}
\end{equation}

Now, we can define the following sets
$$ A:= R_{\bar \lambda + \varepsilon} (B_\delta(0)  \cup  B_\delta(x_0))  \quad \text{and} \quad B:= \Sigma_{\bar \lambda + \varepsilon} \setminus B_{\tilde R}(0),$$
so that $\text{supp}[(\hat u - \hat u_{\bar \lambda + \varepsilon})^+], \ \text{supp}[(\hat v - \hat v_{\bar \lambda + \varepsilon})^+] \subset A \cup B$. Arguing as in \textit{Step 1}, from \eqref{eq:I1} to \eqref{eq:diffv2} it is possible to show that
\begin{equation} \label{B}
	\int_{A \cup B}|\nabla \xi^+_{\bar \lambda + \varepsilon}|^2  \, dx \leq 0 \quad \text{and} \quad \int_{A \cup B}|\nabla \zeta^+_{\bar \lambda + \varepsilon}|^2  \, dx \leq 0,
\end{equation}
where, in  the set $A$, we  argue as in the case of bounded domains, see \cite{BN}. Note that 
 Lemma \ref{lem:positivity}  ensures that we are working far from zero since $\hat u$ and $\hat v$
 are bounded away from zero near the origin and near the translation point $x_0$. This allows us to apply the weak comparison principle in small domains (see e.g.~\cite{esposito, EFS, sciunzi}), and deduce that
\begin{equation} \label{A}
	\int_{A}|\nabla \xi^+_{\bar \lambda + \varepsilon}|^2  \, dx \leq 0 \quad \text{and} \quad \int_{A}|\nabla \zeta^+_{\bar \lambda + \varepsilon}|^2  \, dx \leq 0.
\end{equation}
Finally, collecting \eqref{eq:compactsupptest}, \eqref{B} and \eqref{A}, we deduce $\text{supp}[(\hat u- \hat u_{\bar \lambda + \varepsilon})^+] = \emptyset$, $\text{supp}[(\hat v- \hat v_{\bar \lambda + \varepsilon})^+] = \emptyset$ and hence
$$\hat u< \hat u_{\bar \lambda + \varepsilon} \ \text{ and } \ \hat v < \hat v_{\bar \lambda + \varepsilon} \text{ in } \Sigma_{\bar \lambda + \varepsilon},$$
which is in contradiction with the definition of $\bar \lambda$. So, $\bar{\lambda}$ must be 0.

\noindent \emph{Step 3: Conclusion.} The symmetry of the Kelvin
transform $(\hat u, \hat v)$ in the $x_1$-direction follows now performing the moving plane method in the opposite direction. By the definition of $(\hat u, \hat v)$ given in \eqref{E:kelv}, the symmetry of $\hat{u}$ and $\hat{v}$  w.r.t.~the hyperplane $\{ x_1=0 \}$ implies the same symmetry of the solution $(u,v)$. The symmetry of any solution of \eqref{eq:doublecriticalsingularequbdddiffgamma} follows now recalling that the translation point $x_0$ is arbitrary. We can repeat the same argument with respect to any fixed direction $\nu \in \R^n$. The last one implies that $(u,v)$ must be radially symmetric about the origin.
	
\end{proof}

\section{Asymptotic behavior of solutions} \label{sec3}

The main contribution of the paper is to show that solutions are unique modulo rescaling, in the same spirit of Teraccini's paper, \cite{terracini}. In order to achieve our goal, we need to study the asymptotic behavior of the solutions to \eqref{eq:doublecriticalsingularequbdd} that we recall for the reader convenience
\begin{equation}\tag{$\mathcal{S}^*$}\label{eq:doublecriticalsingularequbddbis}
	\begin{cases}
		\displaystyle -\Delta u\,=\gamma \frac{u}{|x|^2} + u^{2^*-1}+ \nu \alpha u^{\alpha-1} v^\beta & \text{in}\quad\R^n \vspace{0.2cm}\\
		\displaystyle -\Delta v\,=\gamma \frac{v}{|x|^2} + v^{2^*-1}+ \nu \beta u^\alpha v^{\beta-1} & \text{in}\quad\R^n\\
		u,v > 0 &  \text{in}\quad\R^n \setminus \{0\}.
	\end{cases}
\end{equation}
For any $\tau>0$,  let us define
\begin{equation}\label{eq:uvtau}
	(u_\tau(x), v_\tau(x)) := \left(|x|^\tau u(x), |x|^\tau v(x) \right).
\end{equation}

In particular, we set 
\begin{equation}\label{eq:tau}
\tau_1=\frac{n-2}{2} - \sqrt{\left( \frac{n-2}{2}\right)^2-\gamma} \quad \text{or} \quad  \tau_2=\frac{n-2}{2} + \sqrt{\left( \frac{n-2}{2}\right)^2-\gamma},
\end{equation}
which are solutions to the algebraic equation
\begin{equation}\label{algebraic}
	\tau^2-(n-2)\tau+\gamma=0.
\end{equation}
This parameters are well-known in the case of a single equation and  appear describing
 the behavior of the solutions at zero and at infinity. An heuristic argument shows that the behavior at zero should be driven by the weight $|x|^{- \tau_1}$, while $|x|^{- \tau_2}$ gives us the behavior at infinity. In any case, with this choice of $\tau$, it is easy to check that $(u_\tau,v_\tau)$ weakly solves
\begin{equation}\tag{$\mathcal{W}^*$}\label{eq:doublecriticalsingularequbddweighted}
	\begin{cases}
		\displaystyle -\Delta_\tau u_\tau:= -\operatorname{div}\left(\frac{\nabla u_\tau}{|x|^{2 \tau}} \right) = \frac{u_\tau^{2^*-1}}{|x|^{2^* \tau}} + \nu \frac{\alpha}{|x|^{2^* \tau}} u_\tau^{\alpha-1}  v_\tau^\beta & \text{in}\quad\R^n \vspace{0.2cm}\\
		\displaystyle -\Delta_\tau v_\tau:= -\operatorname{div}\left(\frac{\nabla v_\tau}{|x|^{2 \tau}} \right) = \frac{v_\tau^{2^*-1}}{|x|^{2^* \tau}} + \nu \frac{\beta}{|x|^{2^* \tau}} u_\tau^{\alpha} v_\tau^{\beta-1} & \text{in}\quad\R^n\\
		u_\tau,v_\tau > 0 &  \text{in}\quad\R^n \setminus \{0\}.
	\end{cases}
\end{equation}
Such equation, in the weak formulation, is well defined in  the weighted Sobolev space $D^{1,2}_\tau(\R^n)$ for  $\tau=\tau_1$ as above. This space has been already exploited in the literature and, after defining it in the classical way \cite{kilpe}, we may look at it as the
 completion (clousure) of $C^{\infty}_c(\R^n)$ under the norm
$$
\|u\|_{D^{1,2}_\tau (\R^n)}=\left(\int_{\mathbb{R}^n} \frac{|\nabla u|^2}{|x|^{2\tau}} \, dx\right)^{1/2}.
$$
The weighted Sobolev inequality in this space is related to the  Caffarelli-Kohn-Nireberg inequality \cite{CKN}; we refer also to \cite{Flin,AFP}.\\

\noindent With such a transformation at hand, we shall exploit the Moser iteration scheme to prove the following:
\begin{prop}\label{prop:Moser}
Let $\tau =\tau_1$ and $(u_\tau,v_\tau) \in D^{1,2}_\tau (\R^n) \times D^{1,2}_\tau (\R^n)$ be the solution to \eqref{eq:doublecriticalsingularequbddweighted}. Then $(u_\tau,v_\tau) \in L^\infty(\R^n) \times L^\infty(\R^n)$.
\end{prop}

\begin{proof}
For any $\eta \geq 1$ and $T>0$ we define the following Lispchitz function
\begin{equation}\label{eq:PHItau}
	\Upsilon(t)=\Upsilon_{\eta,T}(t):=\begin{cases}
		t^\eta, & \text{if } 0 \leq t \leq T\\
		\eta T^{\eta-1 } (t-T)+T^\eta, & \text{if } t > T.
	\end{cases}
\end{equation}
Making easy computations one can check that for any $w \in D^{1,2}_\tau (\R^n)$ it holds the following inequality in the weak distributional meaning
\begin{equation}\label{eq:convexity}
-\Delta_\tau \Upsilon(w) \leq - \Upsilon'(w) \Delta_\tau w. 
\end{equation}
Now we make all the computations on the first equation of system \eqref{eq:doublecriticalsingularequbddweighted}.
We remark that, since $\Upsilon$ is a Lipschitz function, then $\Upsilon(u_\tau), \Upsilon(v_\tau) \in D^{1,2}_\tau (\R^n)$. Hence, it holds the Sobolev inequality
\begin{equation}\label{Sobolev}
	\left(\int_{\R^n} \frac{|\Upsilon(w)|^{2^*}}{|x|^{2^*\tau}} \,dx\right)^{\frac{1}{2^*}} \leq \mathcal{C}_S \left(\int_{\R^n} \frac{|\nabla \Upsilon(w)|^2}{|x|^{2\tau}} \, dx\right)^{\frac{1}{2}} \qquad \forall w \in D^{1,2}_\tau(\R^n).
\end{equation}
By \eqref{eq:convexity}, we deduce
\begin{equation}\label{eq:Moser1}
	\begin{split}
	\int_{\R^n} \Upsilon(u_\tau) (- \Delta_\tau \Upsilon(u_\tau)) \, dx &\leq \int_{\R^n} \Upsilon(u_\tau) \Upsilon'(u_\tau) (- \Delta_\tau u_\tau) \,dx\\
	&= \int_{\R^n} \Upsilon(u_\tau) \Upsilon'(u_\tau) \left(\frac{u_\tau^{2^*-1}}{|x|^{2^* \tau}} + \nu \frac{\alpha}{|x|^{2^* \tau}} u_\tau^{\alpha-1}  v_\tau^\beta\right) \,dx\\
	&\leq \eta \int_{\R^n} \Upsilon^2(u_\tau)  \left(\frac{u_\tau^{2^*-2}}{|x|^{2^* \tau}} + \nu \frac{\alpha}{|x|^{2^* \tau}} u_\tau^{\alpha-2}  v_\tau^\beta\right) \,dx,
	\end{split}
\end{equation}
where in the last inequality we used the estimate $t \Upsilon'(t) \leq \eta \Upsilon(t)$ for every $t>0$. Integrating by parts in the left hand side of \eqref{eq:Moser1}, and using \eqref{Sobolev} we deduce
\begin{equation}\label{eq:Moser2}
	\left(\int_{\R^n} \frac{|\Upsilon(u_\tau)|^{2^*}}{|x|^{2^*\tau}} \,dx\right)^{\frac{2}{2^*}} \leq \eta \mathcal{C}_S \int_{\R^n} \Upsilon^2(u_\tau)  \left(\frac{u_\tau^{2^*-2}}{|x|^{2^* \tau}} + \nu \frac{\alpha}{|x|^{2^* \tau}} u_\tau^{\alpha-2}  v_\tau^\beta\right) \,dx.
\end{equation}
Arguing in the same way with the second equation of \eqref{eq:doublecriticalsingularequbddweighted}, we obtain
\begin{equation}\label{eq:Moser2'}
	\left(\int_{\R^n} \frac{|\Upsilon(v_\tau)|^{2^*}}{|x|^{2^*\tau}} \,dx\right)^{\frac{2}{2^*}} \leq \eta \mathcal{C}_S \int_{\R^n} \Upsilon^2(v_\tau)  \left(\frac{v_\tau^{2^*-2}}{|x|^{2^* \tau}} + \nu \frac{\beta}{|x|^{2^* \tau}} u_\tau^{\alpha}  v_\tau^{\beta-2}\right) \,dx.
\end{equation}
Our aim is to apply the Moser's iteration scheme (see \cite{moser}). In order to do this, we take into account that $u_\tau,v_\tau \in D^{1,2}_\tau(\R^n)$. Now, let 
$$\eta:=\frac{2^*}{2}$$
and $m_u, m_v \in \R^+$ to be chosen later. We claim to give an estimate of the right hand side of \eqref{eq:Moser2}. We point out that it is easy to check that the function $g(t):=\Upsilon(t)t^{\alpha-2}$ is non-decreasing for every $t \geq 0$ and for every $\alpha>1$. Hence, taking into account this fact and using the definition of $\Upsilon$ in \eqref{eq:PHItau}, we have

\begin{equation}\label{eq:Moser3}
	\begin{split}
		\eta \mathcal{C}_S \int_{\R^n} \Upsilon^2(u_\tau)  &\frac{u_\tau^{2^*-2}}{|x|^{2^* \tau}} \, dx \, + \, \eta \mathcal{C}_S \nu \alpha \int_{\R^n} \frac{\Upsilon^2(u_\tau) u_\tau^{\alpha-2} v_\tau^\beta}{|x|^{2^* \tau }}     \,dx\\
		=& \eta \mathcal{C}_S \int_{\R^n \cap \{u_\tau \leq m_u\}} \Upsilon^2(u_\tau)  \frac{u_\tau^{2^*-2}}{|x|^{2^* \tau}} \, dx + \eta \mathcal{C}_S \int_{\R^n \cap \{u_\tau > m_u\}} \frac{\Upsilon^2(u_\tau)}{|x|^{2\tau}}  \frac{u_\tau^{2^*-2}}{|x|^{(2^*-2) \tau}} \, dx\\
		&+\eta \mathcal{C}_S \nu \alpha \left[\int_{\R^n \cap\{u_\tau \leq v_\tau\}} \frac{\Upsilon^2(u_\tau) u_\tau^{\alpha-2} v_\tau^\beta}{|x|^{2^*\tau}}    + \int_{\R^n \cap\{u_\tau > v_\tau\}} \frac{\Upsilon^2(u_\tau) u_\tau^{\alpha-2} v_\tau^\beta}{|x|^{2^*\tau}}    \,dx \right]\\
		\leq & \eta \mathcal{C}_S \int_{\R^n \cap \{u_\tau \leq m_u\}} \Upsilon^2(u_\tau)  \frac{u_\tau^{2^*-2}}{|x|^{2^* \tau}} \, dx + \eta \mathcal{C}_S \int_{\R^n \cap \{u_\tau > m_u\}} \frac{\Upsilon^2(u_\tau)}{|x|^{2\tau}}  \frac{u_\tau^{2^*-2}}{|x|^{(2^*-2) \tau}} \, dx\\
		&+\eta \mathcal{C}_S \nu \alpha \left[\int_{\R^n \cap\{u_\tau \leq v_\tau\}} \frac{\Upsilon^2(v_\tau) v_\tau^{2^*-2}}{|x|^{2^*\tau}}    + \int_{\R^n \cap\{u_\tau > v_\tau\}} \frac{\Upsilon^2(u_\tau) u_\tau^{2^*-2}}{|x|^{2^*\tau}}    \,dx \right]\\
		\leq & \eta \mathcal{C}_S \int_{\R^n \cap \{u_\tau \leq m_u\}} \Upsilon^2(u_\tau)  \frac{u_\tau^{2^*-2}}{|x|^{2^* \tau}} \, dx + \eta \mathcal{C}_S \int_{\R^n \cap \{u_\tau > m_u\}} \frac{\Upsilon^2(u_\tau)}{|x|^{2\tau}}  \frac{u_\tau^{2^*-2}}{|x|^{(2^*-2) \tau}} \, dx\\
		&+\eta \mathcal{C}_S \nu \alpha \left[\int_{\R^n} \frac{\Upsilon^2(v_\tau) v_\tau^{2^*-2}}{|x|^{2^*\tau}}    + \int_{\R^n} \frac{\Upsilon^2(u_\tau) u_\tau^{2^*-2}}{|x|^{2^*\tau}}    \,dx \right]\,.
\end{split}
\end{equation}		
Using H\"older's inequality with conjugate exponents $\left(\frac{2^*}{2}, \frac{2^*}{2^*-2}\right)$, we get
\begin{equation}\label{eq:Moser3'}
	\begin{split}
		\eta \mathcal{C}_S \int_{\R^n} \Upsilon^2(u_\tau)  &\frac{u_\tau^{2^*-2}}{|x|^{2^* \tau}} \, dx \, + \, \eta \mathcal{C}_S \nu \alpha \int_{\R^n} \frac{\Upsilon^2(u_\tau) u_\tau^{\alpha-2} v_\tau^\beta}{|x|^{2^* \tau }}     \,dx\\		
		\leq&  \eta \mathcal{C}_S (1+\nu \alpha) \int_{\R^n \cap \{u_\tau \leq m_u\}} \Upsilon^2(u_\tau)  \frac{m_u^{2^*-2}}{|x|^{2^* \tau}} \, dx\\
		&+ \eta \mathcal{C}_S (1+\nu \alpha)\left(\int_{\R^n \cap \{u_\tau \geq m_u\}} \frac{\Upsilon^{2^*}(u_\tau)}{|x|^{2^*\tau}} \,dx\right)^{\frac{2}{2^*}}  \left(\int_{\R^n \cap \{u_\tau \geq m_u\}} \frac{u_\tau^{2^*}}{|x|^{2^* \tau}} \, dx\right)^\frac{2^*-2}{2^*}\\
		&+\eta \mathcal{C}_S \nu \alpha \int_{\R^n \cap \{v_\tau \leq m_v\}} \Upsilon^2(v_\tau)  \frac{m_v^{2^*-2}}{|x|^{2^* \tau}} \, dx\\
		&+ \eta \mathcal{C}_S  \nu \alpha\left(\int_{\R^n \cap \{v_\tau \geq m_v\}} \frac{\Upsilon^{2^*}(v_\tau)}{|x|^{2^*\tau}} \,dx\right)^{\frac{2}{2^*}}  \left(\int_{\R^n \cap \{v_\tau \geq m_v\}} \frac{v_\tau^{2^*}}{|x|^{2^* \tau}} \, dx\right)^\frac{2^*-2}{2^*}.
	\end{split}
\end{equation}
 Arguing in an analogous way to \eqref{eq:Moser3} for the second equation of \eqref{eq:doublecriticalsingularequbddweighted}, we are able to deduce   

\begin{equation}\label{eq:Moser4}
	\begin{split}
	\eta \mathcal{C}_S \int_{\R^n} \Upsilon^2(v_\tau)  &\frac{v_\tau^{2^*-2}}{|x|^{2^* \tau}} \, dx \, + \, \eta \mathcal{C}_S \nu \beta \int_{\R^n} \frac{\Upsilon^2(v_\tau) u_\tau^{\alpha} v_\tau^{\beta-2}}{|x|^{2^* \tau }}     \,dx\\
	\leq&  \eta \mathcal{C}_S (1+\nu \beta) \int_{\R^n \cap \{v_\tau \leq m_v\}} \Upsilon^2(v_\tau)  \frac{m_v^{2^*-2}}{|x|^{2^* \tau}} \, dx\\
	&+ \eta \mathcal{C}_S (1+\nu \beta)\left(\int_{\R^n \cap \{v_\tau \geq m_v\}} \frac{\Upsilon^{2^*}(v_\tau)}{|x|^{2^*\tau}} \,dx\right)^{\frac{2}{2^*}}  \left(\int_{\R^n \cap \{v_\tau \geq m_v\}} \frac{v_\tau^{2^*}}{|x|^{2^* \tau}} \, dx\right)^\frac{2^*-2}{2^*}\\
	&+\eta \mathcal{C}_S \nu \beta \int_{\R^n \cap \{u_\tau \leq m_u\}} \Upsilon^2(u_\tau)  \frac{m_u^{2^*-2}}{|x|^{2^* \tau}} \, dx\\
	&+ \eta \mathcal{C}_S  \nu \beta \left(\int_{\R^n \cap \{u_\tau \geq m_u\}} \frac{\Upsilon^{2^*}(u_\tau)}{|x|^{2^*\tau}} \,dx\right)^{\frac{2}{2^*}}  \left(\int_{\R^n \cap \{u_\tau \geq m_u\}} \frac{u_\tau^{2^*}}{|x|^{2^* \tau}} \, dx\right)^\frac{2^*-2}{2^*}.
\end{split}
\end{equation}
Hence, thanks to \eqref{eq:Moser3'}, by \eqref{eq:Moser2} we obtain
\begin{equation} \label{eq:Moser5}
	\begin{split}
		&\left(\int_{\R^n}  \frac{|\Upsilon(u_\tau)|^{2^*}}{|x|^{2^*\tau}} \,dx\right)^{\frac{2}{2^*}} \leq \eta \mathcal{C}_S (1+\nu \alpha) \int_{\R^n \cap \{u_\tau \leq m_u\}} \Upsilon^2(u_\tau)  \frac{m_u^{2^*-2}}{|x|^{2^* \tau}} \, dx\\
		&\qquad + \eta \mathcal{C}_S (1+\nu \alpha)\left(\int_{\R^n \cap \{u_\tau \geq m_u\}} \frac{\Upsilon^{2^*}(u_\tau)}{|x|^{2^*\tau}} \,dx\right)^{\frac{2}{2^*}}  \left(\int_{\R^n \cap \{u_\tau \geq m_u\}} \frac{u_\tau^{2^*}}{|x|^{2^* \tau}} \, dx\right)^\frac{2^*-2}{2^*}\\
		&\qquad +\eta \mathcal{C}_S \nu \alpha \int_{\R^n \cap \{v_\tau \leq m_v\}} \Upsilon^2(v_\tau)  \frac{m_v^{2^*-2}}{|x|^{2^* \tau}} \, dx\\
		&\qquad + \eta \mathcal{C}_S  \nu \alpha\left(\int_{\R^n \cap \{v_\tau \geq m_v\}} \frac{\Upsilon^{2^*}(v_\tau)}{|x|^{2^*\tau}} \,dx\right)^{\frac{2}{2^*}}  \left(\int_{\R^n \cap \{v_\tau \geq m_v\}} \frac{v_\tau^{2^*}}{|x|^{2^* \tau}} \, dx\right)^\frac{2^*-2}{2^*}.
	\end{split}
\end{equation}
Arguing in a similar way with $v_\tau$, thanks to \eqref{eq:Moser4}, by \eqref{eq:Moser2'} we deduce
\begin{equation} \label{eq:Moser6}
	\begin{split}
		&\left(\int_{\R^n} \frac{|\Upsilon(v_\tau)|^{2^*}}{|x|^{2^*\tau}} \,dx\right)^{\frac{2}{2^*}} \leq  \eta \mathcal{C}_S (1+\nu \beta) \int_{\R^n \cap \{v_\tau \leq m_v\}} \Upsilon^2(v_\tau)  \frac{m_v^{2^*-2}}{|x|^{2^* \tau}} \, dx\\
		&\qquad + \eta \mathcal{C}_S (1+\nu \beta)\left(\int_{\R^n \cap \{v_\tau \geq m_v\}} \frac{\Upsilon^{2^*}(v_\tau)}{|x|^{2^*\tau}} \,dx\right)^{\frac{2}{2^*}}  \left(\int_{\R^n \cap \{v_\tau \geq m_v\}} \frac{v_\tau^{2^*}}{|x|^{2^* \tau}} \, dx\right)^\frac{2^*-2}{2^*}\\
		&\qquad +\eta \mathcal{C}_S \nu \beta \int_{\R^n \cap \{u_\tau \leq m_u\}} \Upsilon^2(u_\tau)  \frac{m_u^{2^*-2}}{|x|^{2^* \tau}} \, dx\\
		&\qquad + \eta \mathcal{C}_S  \nu \beta \left(\int_{\R^n \cap \{u_\tau \geq m_u\}} \frac{\Upsilon^{2^*}(u_\tau)}{|x|^{2^*\tau}} \,dx\right)^{\frac{2}{2^*}}  \left(\int_{\R^n \cap \{u_\tau \geq m_u\}} \frac{u_\tau^{2^*}}{|x|^{2^* \tau}} \, dx\right)^\frac{2^*-2}{2^*}.
	\end{split}
\end{equation}
Summing  both \eqref{eq:Moser5} and \eqref{eq:Moser6}, we deduce that
\begin{equation} \label{eq:Moser7}
	\begin{split}
	&\left(\int_{\R^n} \frac{|\Upsilon(u_\tau)|^{2^*}}{|x|^{2^*\tau}} \,dx \right)^{\frac{2}{2^*}} + 	\left(\int_{\R^n} \frac{|\Upsilon(v_\tau)|^{2^*}}{|x|^{2^*\tau}} \,dx\right)^{\frac{2}{2^*}}\\
	&\ \ \leq  \eta \mathcal{C}_S(1+ 2^*\nu) \left[\int_{\R^n \cap \{u_\tau \leq m_u\}} \Upsilon^2(u_\tau)  \frac{m_u^{2^*-2}}{|x|^{2^* \tau}} \, dx +\int_{\R^n \cap \{v_\tau \leq m_v\}} \Upsilon^2(v_\tau)  \frac{m_v^{2^*-2}}{|x|^{2^* \tau}} \, dx\right]\\
	&\qquad + \eta \mathcal{C}_S (1+ 2^*\nu)\left(\int_{\R^n \cap \{u_\tau \geq m_u\}} \frac{\Upsilon^{2^*}(u_\tau)}{|x|^{2^*\tau}} \,dx\right)^{\frac{2}{2^*}}  \left(\int_{\R^n \cap \{u_\tau \geq m_u\}} \frac{u_\tau^{2^*}}{|x|^{2^* \tau}} \, dx\right)^\frac{2^*-2}{2^*}\\
	&\qquad + \eta \mathcal{C}_S (1+ 2^*\nu) \left(\int_{\R^n \cap \{v_\tau \geq m_v\}} \frac{\Upsilon^{2^*}(v_\tau)}{|x|^{2^*\tau}} \,dx\right)^{\frac{2}{2^*}}  \left(\int_{\R^n \cap \{v_\tau \geq m_v\}} \frac{v_\tau^{2^*}}{|x|^{2^* \tau}} \, dx\right)^\frac{2^*-2}{2^*}.
\end{split}
\end{equation}
Since $u_\tau, v_\tau \in D^{1,2}_\tau(\R^n)$, we can fix $m_u,m_v$ such that
\begin{equation}
	\left(\int_{\R^n \cap \{u_\tau \geq m_u\}} \frac{u_\tau^{2^*}}{|x|^{2^* \tau}} \, dx\right)^\frac{2^*-2}{2^*} , \left(\int_{\R^n \cap \{v_\tau \geq m_v\}} \frac{v_\tau^{2^*}}{|x|^{2^* \tau}} \, dx\right)^\frac{2^*-2}{2^*} \leq \frac{1}{2\eta \mathcal{C}_S (1+ 2^*\nu)} 
\end{equation}
Hence, from \eqref{eq:Moser7} we obtain
\begin{equation} \label{eq:Moser8}
\begin{split}
		&\left(\int_{\R^n} \frac{|\Upsilon(u_\tau)|^{2^*}}{|x|^{2^*\tau}} \,dx \right)^{\frac{2}{2^*}} + 	\left(\int_{\R^n} \frac{|\Upsilon(v_\tau)|^{2^*}}{|x|^{2^*\tau}} \,dx\right)^{\frac{2}{2^*}}\\
		&\ \ \leq  \eta \mathcal{C}_S(1+ 2^*\nu) \left[m_u^{2^*-2}\int_{\R^n}  \frac{\Upsilon^2(u_\tau) }{|x|^{2^* \tau}} \, dx +m_v^{2^*-2}\int_{\R^n}   \frac{\Upsilon^2(v_\tau)}{|x|^{2^* \tau}} \, dx\right]\\
		&\ \ \leq  \eta C \left[\int_{\R^n}  \frac{u_\tau^{2\eta}}{|x|^{2^* \tau}} \, dx +\int_{\R^n}   \frac{v_\tau^{2\eta}}{|x|^{2^* \tau}} \, dx\right],
\end{split}
\end{equation}
where $C$ depends on $m_u,m_v,\nu,n$ and $\mathcal{C}_S$. Moreover, we used that $\Upsilon(t) \leq t^\eta$ and that $u_\tau, v_\tau \in D^{1,2}_\tau(\R^n)$. Hence, recalling that $\eta=\frac{2^*}{2}$ and taking $T \rightarrow +\infty$, thanks also to Fatou's lemma one has
\begin{equation} \label{eq:Moser9-}
\left(\int_{\R^n} \frac{u_\tau^{2^*\eta}}{|x|^{2^*\tau}} \,dx \right)^{\frac{2}{2^*}} + 	\left(\int_{\R^n} \frac{v_\tau^{2^*\eta}}{|x|^{2^*\tau}} \,dx\right)^{\frac{2}{2^*}} < +\infty.
\end{equation}
Moreover, since $\frac{2}{2^*}<1$, we also deduce
\begin{equation} \label{eq:Moser9}
	\left(\int_{\R^n} \frac{u_\tau^{2^*\eta}}{|x|^{2^*\tau}} \,dx   + 	 \int_{\R^n} \frac{v_\tau^{2^*\eta}}{|x|^{2^*\tau}} \,dx\right)^{\frac{2}{2^*}}  \leq \left(\int_{\R^n} \frac{u_\tau^{2^*\eta}}{|x|^{2^*\tau}} \,dx \right)^{\frac{2}{2^*}} + 	\left(\int_{\R^n} \frac{v_\tau^{2^*\eta}}{|x|^{2^*\tau}} \,dx\right)^{\frac{2}{2^*}} < +\infty.
\end{equation}
Now, we are able to apply the Moser's iteration scheme, in order to prove our result. For any $k \geq 1$, let us define the sequence $\{\eta_k\}$ by
$$2 \eta_{k+1} + 2^*-2=2^*\eta_k,$$
and we set $\eta_1 := \frac{2^*}{2}$. The choice of $\eta_k$ at each step is clear in the first term of the right hand side of \eqref{eq:Moser3}. Then starting again by \eqref{eq:Moser3} and \eqref{eq:Moser4}, iterating we deduce that
\begin{equation} \label{eq:Moser10}
	\begin{split}
		&\left(\int_{\R^n} \frac{u_\tau^{2^*\eta_{k+1}}}{|x|^{2^*\tau}} \,dx + \int_{\R^n} \frac{v_\tau^{2^*\eta_{k+1}}}{|x|^{2^*\tau}} \,dx\right)^{\frac{1}{2^*(\eta_{k+1}-1)}}\\
		&\qquad \leq   (C \eta_{k+1})^{\frac{1}{2(\eta_{k+1}-1)}} \left(\int_{\R^n}  \frac{u_\tau^{2^*\eta_k}}{|x|^{2^* \tau}} \, dx + \int_{\R^n}   \frac{v_\tau^{2^*\eta_k}}{|x|^{2^* \tau}} \, dx\right)^{\frac{1}{2^*(\eta_k-1)}}.
	\end{split}
\end{equation}
Finally, we set
\begin{equation} \label{settingMoser}
	\mathcal{I}_k := \left(\int_{\R^n}  \frac{u_\tau^{2^*\eta_k}}{|x|^{2^* \tau}} \, dx + \int_{\R^n}   \frac{v_\tau^{2^*\eta_k}}{|x|^{2^* \tau}} \, dx\right)^{\frac{1}{2^*(\eta_k-1)}}, \ \ \mathcal{C}_k:=(C \eta_{k+1})^{\frac{1}{2(\eta_{k+1}-1)}},
\end{equation}
and we obtain the recursive inequality 
$$\mathcal{I}_{k+1} \leq \mathcal{C}_k \mathcal{I}_k, \ \ k \geq 1.$$ By induction we can easily deduce that
\begin{equation}\label{eq:LogMoser}
\begin{split}
	\log \mathcal{I}_{k+1} &\leq \sum_{j=2}^{k+1} \log\mathcal{C}_j  + \log \mathcal{I}_1\leq \sum_{j=2}^{+\infty} \log\mathcal{C}_j  + \log \mathcal{I}_1 < +\infty,
\end{split}
\end{equation}
where in the last inequality we used that the series $\sum_{j=2}^{+\infty} \log \mathcal{C}_j$ converges. Hence, thanks to the monotonicity of the integral and by \eqref{eq:LogMoser}, for every $R>0$, we have
\begin{equation}
\log \left(\left(\int_{B_R}  \frac{u_\tau^{2^*\eta_k}}{|x|^{2^* \tau}} \, dx + \int_{B_R}   \frac{v_\tau^{2^*\eta_k}}{|x|^{2^* \tau}} \, dx\right)^{\frac{1}{2^*(\eta_k-1)}}\right) \leq C.
\end{equation}
Since $|x|<R$, we have
\begin{equation}
\begin{split}
	\frac{\tau}{\eta_k-1} \log \left(\frac{1}{R}\right) &+ \log \left(\left(\int_{B_R}  u_\tau^{2^*\eta_k} \, dx \right)^{\frac{1}{2^*(\eta_k-1)}}\right)\\
	&\leq \frac{\tau}{\eta_k-1} \log \left(\frac{1}{R}\right) + \log \left(\left(\int_{B_R}  u_\tau^{2^*\eta_k} \, dx + \int_{B_R}   v_\tau^{2^*\eta_k} \, dx\right)^{\frac{1}{2^*(\eta_k-1)}}\right) \leq C.
\end{split}
\end{equation}
Passing to the limit for $k \rightarrow + \infty$ and thus $\eta_{k} \rightarrow + \infty$,  we deduce that
$$(u_\tau,v_\tau) \in L^\infty(B_R) \times L^\infty(B_R).$$
The thesis follows by the arbitrariness of $R$.

\end{proof}
%
%
In the previous section we showed that any solution $(u,v) \in H^1_{loc}(\R^n) \times H^1_{loc}(\R^n)$ to \eqref{eq:doublecriticalsingularequbdd} is radial, and, hence,  \eqref{eq:doublecriticalsingularequbddbis}  becomes 
\begin{equation} \tag{$\mathcal{R}_{\mathcal{S}}^*$}\label{eq:RadialCritProb}
	\begin{cases}
		\displaystyle (r^{n-1}u')'+r^{n-1}\left(\gamma \frac{u}{r^2} + u^{2^*-1}+ \nu \alpha u^{\alpha-1} v^\beta\right)=0 & \text{in}\quad (0,+\infty)\\
		\displaystyle \displaystyle (r^{n-1}v')'+r^{n-1}\left(\gamma \frac{v}{r^2} + v^{2^*-1}+ \nu \beta u^\alpha v^{\beta-1}\right)=0 & \text{in}\quad (0,+\infty)\\
		u,v > 0 &  \text{in}\quad (0,+\infty).
	\end{cases}
\end{equation} 
Obviously also $u_\tau,v_\tau$ are radial and they satisfy\\
\begin{equation} \tag{$\mathcal{R}_{\mathcal{W}}^*$}\label{eq:RadialCritProbbis}
	\begin{cases}
		\displaystyle (r^{n-1-2\tau}u_\tau')'+ r^{n-1-2^*\tau} \left(u_\tau^{2^*-1}+ \nu \alpha u_\tau^{\alpha-1} v_\tau^\beta\right)=0 & \text{in}\quad (0,+\infty)\\
		\displaystyle \displaystyle (r^{n-1-2\tau}v_\tau')'+ r^{n-1-2^*\tau} \left(v_\tau^{2^*-1}+ \nu \beta u_\tau^\alpha v_\tau^{\beta-1}\right)=0 & \text{in}\quad (0,+\infty)\\
		u,v > 0 &  \text{in}\quad (0,+\infty).
	\end{cases}
\end{equation} 
Thanks to to Proposition \ref{prop:Moser}, we deduce that 
\begin{center}
	$|x|^{\tau_1}\cdot u(x)$ and $|x|^{\tau_1} \cdot v(x)$ are bounded functions in $\R^n$\,.
\end{center}
This is a first information regarding the behavior of the solutions at zero (the behavior that we get at infinity is not sharp). In any case we are in position to get a complete picture of the asymptotic behavior of the solutions by proving the following:

\begin{prop}\label{prop:asBehav}
Let $(u,v) \in D^{1,2}(\R^n) \times D^{1,2}(\R^n)$ be any solution to \eqref{eq:doublecriticalsingularequbdd}. Then, we have
\begin{equation} \label{eq:behaviorat0}
	\lim_{|x| \rightarrow 0^+} |x|^{\tau_1} \cdot  u(x) =:u_0, \qquad \lim_{|x| \rightarrow 0^+} |x|^{\tau_1} \cdot v(x) =:v_0,
\end{equation}
where $u_0, v_0 \in (0,+\infty)$.
Moreover, we have
\begin{equation}\label{eq:behavioratinfty}
	\lim_{|x| \rightarrow +\infty} |x|^{\tau_2} \cdot  u(x) =:u_\infty, \qquad \lim_{|x| \rightarrow 0^+} |x|^{\tau_2} \cdot v(x) =:v_\infty,
\end{equation}
where $u_\infty, v_\infty \in (0,+\infty)$.
\end{prop}
\begin{proof}
Let $r:=|x|$. By Problem \eqref{eq:RadialCritProbbis} we deduce that both $r^{n-1-2\tau}u_\tau'(r)$ and $r^{n-1-2\tau}v_\tau'(r)$ are strictly decreasing functions. Therefore, we get $r^{n-1-2\tau}u_{\tau_1}'(r)$ and $r^{n-1-2\tau}v_{\tau_1}'(r)$ have a sign near zero, and this immediately implies that $u_{\tau_1}'(r),v_{\tau_1}'(r)$ have a sign near zero. Hence $u_{\tau_1}, v_{\tau_1}$ are monotone  bounded functions and then they must admit limit $u_0,v_0 \in [0,+\infty)$ as $r \rightarrow 0^+$. By the strong maximum principle (see e.g. \cite{ACP}) for the system
$$\begin{cases}
	-\Delta_{\tau_1} u_{\tau_1} \geq 0 & \text{in } B_\delta(0)\\
	-\Delta_{\tau_1} v_{\tau_1} \geq 0 & \text{in } B_\delta(0)\\
	u_{\tau_1},v_{\tau_1} \geq 0  & \text{in } B_\delta(0),
\end{cases}$$
since $u_{\tau_1}, v_{\tau_1} > 0$ in $B_\delta(0) \setminus \{0\}$, we get $u_{\tau_1}, v_{\tau_1} > 0$ in $B_\delta(0)$ and hence $u_0,v_0>0$. Now, we have in force  the estimate as $r \rightarrow 0^+$, i.e.~\eqref{eq:behaviorat0} holds true. Since $(u,v) \in D^{1,2}(\R^n) \times D^{1,2}(\R^n)$, it is standard to prove that the Kelvin transformations $(\hat u, \hat v) \in D^{1,2}(\R^n) \times D^{1,2}(\R^n)$. Hence, we get
\begin{equation} 
		\lim_{|x| \rightarrow 0^+} |x|^{\tau_1} \cdot  \hat u\left(\frac{x}{|x|^2}\right) = \hat u_0, \qquad \lim_{|x| \rightarrow 0^+} |x|^{\tau_1} \cdot \hat v\left(\frac{x}{|x|^2}\right) = \hat v_0,
\end{equation}
which implies  \eqref{eq:behaviorat0}. By the last one, if we change variable $y:=x/|x|^2$ and use the identity $\tau_1+\tau_2=n-2$, we deduce
\begin{equation} 
\begin{split}
	\hat u_0&=\lim_{|x| \rightarrow 0^+} |x|^{\tau_1} \cdot  \hat u(x) = \lim_{|y| \rightarrow +\infty} |y|^{-\tau_1} \cdot |y|^{n-2}u(y) = \lim_{|y| \rightarrow +\infty} |y|^{\tau_2} \cdot u(y),\\
	\hat v_0&=\lim_{|x| \rightarrow 0^+} |x|^{\tau_1} \cdot \hat v(x) = \lim_{|y| \rightarrow +\infty} |y|^{-\tau_1} \cdot |y|^{n-2} v(y) = \lim_{|y| \rightarrow +\infty} |y|^{\tau_2} \cdot v(y),
\end{split}
\end{equation}
which implies \eqref{eq:behavioratinfty}.

\end{proof}

%

\section{Uniqueness and classification of solutions} \label{sec4}

The aim of this section is to classify solutions to \eqref{eq:doublecriticalsingularequbddbis}  up to the associated rescaling.
Given  $(u,v)$  a solution to the problem \eqref{eq:RadialCritProb}, consider 
\begin{equation} \label{eq:newVariables}
\begin{split}
	y_u(t):=r^\delta u(r) \quad &\text{and} \quad y_v(t):=r^\delta v(r),\\
\end{split}
\end{equation}
where $t:=\log r$ with $r>0$ and $\delta=\frac{n-2}{2}$. By direct computation, it is possible to deduce that system \eqref{eq:RadialCritProb} is equivalent to 
\begin{equation}
	\begin{cases} \tag{$\mathcal{R}_*'$}\label{eq:RadialCritProbBis}
		y''_u-(\delta^2-\gamma) y_u + y_u^{2^*-1} + \nu \alpha y_u^{\alpha-1} y_v^\beta=0 & \text{in}\quad \R\\
		y''_v-(\delta^2-\gamma) y_v + y_v^{2^*-1} + \nu \beta y_u^\alpha y_v^{\beta-1}=0 & \text{in}\quad \R\\ 
		y_u,y_v>0 &\text{in}\quad \R.
	\end{cases}
\end{equation}

Next, we will prove that $y_u$ and $y_v$ reach a unique maximum point, which actually is the same for both components. To show this fact, we will use the asymptotic behavior of $y_u$ and $y_v$.

\begin{lem}\label{lemasymptot}
Let $y_u,y_v$ be the functions defined in \eqref{eq:newVariables}, they satisfy
$$
\lim_{t \rightarrow -\infty} y_u(t)= 0, \quad \lim_{t \rightarrow +\infty} y_u(t)= 0, \quad \lim_{t \rightarrow -\infty} y_v(t)= 0, \quad \lim_{t \rightarrow +\infty} y_v(t)= 0.
$$
Moreover, there exist $\bar c_u, \bar C_u, \underline{c}_u, \underline{C}_u, \bar c_v, \bar C_v, \underline{c}_v, \underline{C}_v, M$ positive constants such that
\begin{eqnarray}
	&& \underline{c}_u e^{(\delta-\tau_1)t} u_{\tau_1}(e^t) \leq y_u(t) \leq \underline{C}_u e^{(\delta-\tau_1)t} u_{\tau_1}(e^t) \text{ for } t \leq -M,\\	
	&& \bar c_u e^{(\delta-\tau_2)t} u_{\tau_2}(e^t) \leq y_u(t) \leq \bar C_u e^{(\delta-\tau_2)t} u_{\tau_2}(e^t) \text{ for } t \geq M,\\
	&& \underline{c}_v e^{(\delta-\tau_1)t} v_{\tau_1}(e^t) \leq y_v(t) \leq \underline{C}_v e^{(\delta-\tau_1)t} v_{\tau_1}(e^t) \text{ for } t \leq -M,\\	
	&& c_u e^{(\delta-\tau_2)t} v_{\tau_2}(e^t) \leq y_v(t) \leq C_v e^{(\delta-\tau_2)t} v_{\tau_2}(e^t) \text{ for } t \geq M,
\end{eqnarray}
where $\tau_1$ and $\tau_2$ were introduced \eqref{eq:tau}. Furthermore, we have that 
\begin{equation}\label{eq:quotient}
	\lim_{t \rightarrow - \infty} \frac{y_u(t)}{y_v(t)} = L_1  \in (0,+\infty) \ \text{ and } \ 	\lim_{t \rightarrow + \infty} \frac{y_u(t)}{y_v(t)} = L_2  \in (0,+\infty).
\end{equation}
\end{lem}
The proof  follows by Proposition \ref{prop:Moser} and \ref{prop:asBehav} .

Let us point out that, unfortunately, one could not apply the ODE approaches of Abdellaoui-Felli-Peral \cite{abdellaoui} and Terracini \cite{terracini} in a straightforward way. On the contrary, we shall prove the following result. 

\begin{lem} \label{lem:extremums}
The functions $y_u$ and $y_v$ admits only a simultaneous global maximum at a point $\tilde{t}\in\mathbb{R}$.
\end{lem}

\begin{proof}
To prove our result we will use the moving plane procedure for solutions to \eqref{eq:RadialCritProbBis}.
For $\lambda \in \R$, let us define the following functions
$$
w_{u,\lambda}(t):=y_u(t)-y_{u,\lambda}(t)=y_u(t) - y_u(2\lambda-t), \quad w_{v,\lambda}(t):=y_v(t)-y_{v,\lambda}(t)=y_v(t) - y_v(2\lambda-t),
$$
Notice that $w_u, w_v$ solves the following system
\begin{equation}
	\begin{cases}\label{eq:RadialCritProbBis2}
			- w''_{u,\lambda} = -(\delta^2-\gamma_1) w_{u,\lambda} + (y_u^{2^*-1}-y_{u,\lambda}^{2^*-1}) + \nu \alpha (y_u^{\alpha-1} y_v^\beta-y_{u,\lambda}^{\alpha-1} y_{v,\lambda}^\beta)  & \text{in}\quad \R\\
			- w''_{v,\lambda} = -(\delta^2-\gamma_2) w_{v,\lambda} + (y_v^{2^*-1}-y_{v,\lambda}^{2^*-1}) + \nu \beta (y_u^{\alpha} y_v^{\beta-1}-y_{u,\lambda}^{\alpha} y_{v,\lambda}^{\beta-1}) & \text{in}\quad \R.
		\end{cases}
	\end{equation}
We pass to the weak formulation of \eqref{eq:RadialCritProbBis2} given by
\begin{equation}
	\begin{split}\label{eq:RadialCritProbBis2Weak}
		\int_\R w'_{u,\lambda} \varphi' \, dt = &-(\delta^2-\gamma)  \int_\R w_{u,\lambda} \varphi \, dt + \int_\R (y_u^{2^*-1}-y_{u,\lambda}^{2^*-1}) \varphi \, dt \\
		&+ \nu \alpha \int_\R (y_u^{\alpha-1} y_v^\beta-y_{u,\lambda}^{\alpha-1} y_{v,\lambda}^\beta) \varphi \, dt  \qquad \forall \varphi \in C^1_c(\R),\\
		\int_\R w'_{v,\lambda} \psi' \, dt =&- \int_\R (\delta^2-\gamma) w_{v,\lambda} \psi \, dt + \int_\R (y_v^{2^*-1}-y_{v,\lambda}^{2^*-1}) \psi' \, dt \\
		&+ \nu \beta \int_\R (y_u^{\alpha} y_v^{\beta-1}-y_{u,\lambda}^{\alpha} y_{v,\lambda}^{\beta-1}) \psi \, dt \qquad \forall \psi \in C^1_c(\R).
	\end{split}
\end{equation}
Thanks to Lemma \ref{lemasymptot} and arguing in a similar way to the proof of Theorem \ref{thm:main1}, by density argument we can choose the following test functions
$$\varphi:= w_{u, \lambda}^+  \qquad \text{and} \qquad \psi:= w_{v, \lambda}^+,$$
Moreover, we point out that $0 \leq w_{u,\lambda}^+ \leq y_u$ and  $0 \leq w_{v,\lambda}^+ \leq y_v$. Putting the test functions into the weak formulations, \eqref{eq:RadialCritProbBis2Weak} becomes
\begin{equation}
	\begin{split}\label{eq:RadialCritProbBis2Weak1}
		\int_{(-\infty, \lambda]} [(w^+_{u,\lambda})']^2  \, dt = & -(\delta^2-\gamma)  \int_{(-\infty, \lambda]} w_{u,\lambda} w^+_{u,\lambda} \, dt + \int_{(-\infty, \lambda]} (y_u^{2^*-1}-y_{u,\lambda}^{2^*-1}) w^+_{u,\lambda} \, dt \\
		&+ \nu \alpha \int_{(-\infty, \lambda]} (y_u^{\alpha-1} y_v^\beta-y_{u,\lambda}^{\alpha-1} y_{v,\lambda}^\beta) w^+_{u,\lambda} \, dt,\\
		\int_{(-\infty, \lambda]} [(w^+_{v,\lambda})']^2  \, dt = & -(\delta^2-\gamma)  \int_{(-\infty, \lambda]} w_{v,\lambda} w^+_{v,\lambda} \, dt + \int_{(-\infty, \lambda]} (y_v^{2^*-1}-y_{v,\lambda}^{2^*-1}) w^+_{v,\lambda} \, dt \\
		&+ \nu \beta \int_\R (y_u^{\alpha} y_v^{\beta-1}-y_{u,\lambda}^{\alpha} y_{v,\lambda}^{\beta-1}) w^+_{v,\lambda}   \, dt.
	\end{split}
\end{equation}
Now, we focus our attention on the first equation, and we consider $\lambda \leq -M$ (where $M$ is given by Lemma \ref{lemasymptot}). In particular, we argue as in the proof of Theorem \ref{thm:main1}. Hence, using the definition of $w_{u,\lambda}^+$, the Mean Value's Theorem and Lemma \ref{lemasymptot} 
\begin{equation}
	\begin{split}\label{eq:MovStep2}
		\int_{(-\infty, \lambda]} [(w^+_{u,\lambda})']^2 \, dt \leq & - (\delta^2-\gamma)  \int_{(-\infty, \lambda]} (w^+_{u,\lambda})^2 \, dt  + \int_{(-\infty, \lambda]} y_u^{2^*-2} (w^+_{u,\lambda})^2 \, dt \\
		&+ \nu \alpha  \int_{(-\infty, \lambda]}  y_{v}^\beta (y_u^{\alpha-1} - y_{u,\lambda}^{\alpha-1}) w^+_{u,\lambda}  \, dt\\
		&+ \nu \alpha \int_{(-\infty, \lambda]} y_{u,\lambda}^{\alpha-1} (y_v^\beta - y_{v, \lambda}^\beta) w^+_{u,\lambda}  \, dt.
	\end{split}
\end{equation}		
Arguing as we did for $\mathcal{I}_5$, we get
\begin{equation}
	\begin{split}\label{eq:MovStep2bis}
		\int_{(-\infty, \lambda]} [(w^+_{u,\lambda})']^2 \, dt \leq & - (\delta^2-\gamma)  \int_{(-\infty, \lambda]} (w^+_{u,\lambda})^2 \, dt  + \int_{(-\infty, \lambda]} y_u^{2^*-2} (w^+_{u,\lambda})^2 \, dt \\
		&+ \nu \alpha  \int_{(-\infty, \lambda]} e^{-(\delta-\tau_1)(\alpha-1)t} y_{v}^\beta \left[(e^{(\delta-\tau_1)t}y_u)^{\alpha-1} - (e^{(\delta-\tau_1)t} y_{u,\lambda})^{\alpha-1}\right] w^+_{u,\lambda}  \, dt\\
		&+ \nu \alpha \int_{(-\infty, \lambda]} y_{u,\lambda}^{\alpha-1} (y_v^\beta - y_{v, \lambda}^\beta) w^+_{u,\lambda}  \, dt\\
		= & - (\delta^2-\gamma)  \int_{(-\infty, \lambda]} (w^+_{u,\lambda})^2 \, dt  + \int_{(-\infty, \lambda]} y_u^{2^*-2} (w^+_{u,\lambda})^2 \, dt \\
		&+ \nu \alpha  \int_{(-\infty, \lambda]} e^{-(\delta-\tau_1)(\alpha-1)t} y_{v}^\beta \frac{(e^{(\delta-\tau_1)t}y_u)^{\alpha-1} - (e^{(\delta-\tau_1)t} y_{u,\lambda})^{\alpha-2}}{e^{(\delta-\tau_1)t}y_u - e^{(\delta-\tau_1)t} y_{u,\lambda}} w_{u,\lambda}  w^+_{u,\lambda}\, dt\\
		&+ \nu \alpha \int_{(-\infty, \lambda]} y_{u,\lambda}^{\alpha-1} (y_v^\beta - y_{v, \lambda}^\beta)  w^+_{u,\lambda}  \, dt.\\
	\end{split}
\end{equation}	
Now, using the same idea contained in Remark \ref{fsdfd}, we get
\begin{equation}
	\begin{split}\label{eq:MovStep2tris}
		\int_{(-\infty, \lambda]} [(w^+_{u,\lambda})']^2 \, dt \leq & - (\delta^2-\gamma)  \int_{(-\infty, \lambda]} (w^+_{u,\lambda})^2 \, dt  + \int_{(-\infty, \lambda]} y_u^{2^*-2} (w^+_{u,\lambda})^2 \, dt \\
		&+ \nu \alpha  \int_{(-\infty, \lambda]} e^{-(\delta-\tau_1)(\alpha-1)t} y_{v}^\beta \left[(e^{(\delta-\tau_1)t}y_u)^{\alpha-1} - (e^{(\delta-\tau_1)t} y_{u,\lambda})^{\alpha-1}\right] w^+_{u,\lambda}  \, dt\\
		&+ \nu \alpha \int_{(-\infty, \lambda]} y_{u,\lambda}^{\alpha-1} (y_v^\beta - y_{v, \lambda}^\beta) w^+_{u,\lambda}  \, dt\\
		\leq & - (\delta^2-\gamma)  \int_{(-\infty, \lambda]} (w^+_{u,\lambda})^2 \, dt  + \int_{(-\infty, \lambda]} y_u^{2^*-2} (w^+_{u,\lambda})^2 \, dt \\
		&+ \nu \alpha  \int_{(-\infty, \lambda]} e^{-(\delta-\tau_1)(\alpha-2)t} y_{v}^\beta \frac{(e^{(\delta-\tau_1)t}y_u)^{\alpha-1} - (e^{(\delta-\tau_1)t} y_{u,\lambda})^{\alpha-1}}{e^{(\delta-\tau_1)t}y_u - e^{(\delta-\tau_1)t} y_{u,\lambda}} w_{u,\lambda}  w^+_{u,\lambda}\, dt\\
		&+ \nu \alpha \int_{(-\infty, \lambda] \cap \{y_v > y_{v,\lambda}\}} y_{u,\lambda}^{\alpha-1} (y_v^\beta - y_{v, \lambda}^\beta)  w^+_{u,\lambda}  \, dt.\\
		\leq & - (\delta^2-\gamma)  \int_{(-\infty, \lambda]} (w^+_{u,\lambda})^2 \, dt  + \int_{(-\infty, \lambda]} y_u^{2^*-2} (w^+_{u,\lambda})^2 \, dt \\
		&+ \nu \alpha  \int_{(-\infty, \lambda]} e^{-(\delta-\tau_1)(\alpha-1)t} y_{v}^\beta h'(\sigma) (w^+_{u,\lambda})^2\, dt\\
		&+ \nu \alpha \beta \int_{(-\infty, \lambda] \cap \{y_v > y_{v,\lambda}\}} y_{u,\lambda}^{\alpha-1} y_v^{\beta-1} w^+_{v,\lambda} w^+_{u,\lambda}  \, dt,\\
	\end{split}
\end{equation}	
where  we always get 
$$\frac{(e^{(\delta-\tau_1)t}y_u)^{\alpha-1} - (e^{(\delta-\tau_1)t} y_{u,\lambda})^{\alpha-1}}{e^{(\delta-\tau_1)t}y_u - e^{(\delta-\tau_1)t} y_{u,\lambda}} \leq C.$$
by Remark \ref{fsdfd}
%
%
Now, we recall that all the integrals in \eqref{eq:MovStep2tris} are computed in  $\text{supp}(w_{u,\lambda}^+) \cap (-\infty, \lambda]$. Hence,  we have $y_{u,\lambda}(t) \leq  y_u(t)$ for every $t \in \text{supp}(w_{u,\lambda}^+) \cap (-\infty, \lambda]$. Moreover, since $\lambda \leq - M$,  by Lemma \ref{lemasymptot} we know that 
$$y_u(t) \leq \underline{C}_u e^{(\delta-\tau_1)t} u_{\tau_1}(e^t) \leq C\underline{C}_u e^{(\delta-\tau_1)t}, \ \ \text{ and that } \ \ y_v(t) \leq \underline{C}_v e^{(\delta-\tau_1)t} v_{\tau_1}(e^t) \leq C\underline{C}_v e^{(\delta-\tau_1)t},$$ 
for any $t \leq - M$. Collecting all these information \eqref{eq:MovStep2tris} becomes
\begin{equation}
	\begin{split}\label{eq:MovStepFinal_ubis}		
		\int_{(-\infty, \lambda]} [(w^+_{u,\lambda})']^2 \, dt \leq & - (\delta^2-\gamma)  \int_{(-\infty, \lambda]} (w^+_{u,\lambda})^2 \, dt  + C \int_{(-\infty, \lambda]} e^{(2^*-2)(\delta-\tau_1)t} (w^+_{u,\lambda})^2 \, dt \\
		&+ \nu \alpha (\alpha -1) C \int_{(-\infty, \lambda]} e^{(2^*-2)(\delta-\tau_1)t} (w^+_{u,\lambda})^2  \, dt\\
		&+ \nu \alpha \beta C \int_{(-\infty, \lambda]} e^{(2^*-2)(\delta-\tau_1)t} w_{v,\lambda}^+ w^+_{u,\lambda}  \, dt.
	\end{split}
\end{equation}
Using Young's inequality in the last term of the right hand side, we deduce
\begin{equation}
	\begin{split}\label{eq:MovStepFinal_utris}		
		\int_{(-\infty, \lambda]} [(w^+_{u,\lambda})']^2 \, dt \leq & - (\delta^2-\gamma)  \int_{(-\infty, \lambda]} (w^+_{u,\lambda})^2 \, dt  + \bar C \int_{(-\infty, \lambda]} e^{(2^*-2)(\delta-\tau_1)t} (w^+_{u,\lambda})^2  \, dt  \\
		&+ \nu \alpha \beta C \int_{(-\infty, \lambda]} e^{(2^*-2)(\delta-\tau_1)t} (w_{v,\lambda}^+)^2   \, dt,
	\end{split}
\end{equation}
where $\bar C$ is a positive constant depending only on $\nu, \alpha$ and $\|u_\tau\|_\infty$.
Arguing in the same way for the second equation of \eqref{eq:RadialCritProbBis2Weak1}, we obtain
\begin{equation}
	\begin{split}\label{eq:MovStepFinal_v}		
		\int_{(-\infty, \lambda]} [(w^+_{v,\lambda})']^2 \, dt \leq & - (\delta^2-\gamma)  \int_{(-\infty, \lambda]} (w^+_{v,\lambda})^2 \, dt  + \tilde C \int_{(-\infty, \lambda]} e^{(2^*-2)(\delta-\tau_1)t} (w^+_{v,\lambda})^2  \, dt  \\
	&+ \nu \alpha \beta C \int_{(-\infty, \lambda]} e^{(2^*-2)(\delta-\tau_1)t} (w_{u,\lambda}^+)^2   \, dt,
	\end{split}
\end{equation}
where $\tilde C$ is a positive constant depending only on $\nu, \alpha$ and $\|u_\tau\|_\infty$. 
Summing both \eqref{eq:MovStepFinal_utris}, \eqref{eq:MovStepFinal_v}, we have
\begin{equation}
	\begin{split}\label{eq:MovStepFinal_u+v}		
		\int_{(-\infty, \lambda]} [(w^+_{u,\lambda})']^2  \, dt &+ \int_{(-\infty, \lambda]} [(w^+_{v,\lambda})']^2  \, dt  \\
		& \leq   - (\delta^2-\gamma)  \int_{(-\infty, \lambda]} (w^+_{u,\lambda})^2 \, dt   - (\delta^2-\gamma)  \int_{(-\infty, \lambda]} (w^+_{v,\lambda})^2 \, dt \\
		& + \bar C \int_{(-\infty, \lambda]} e^{(2^*-2)(\delta-\tau_1)t} (w^+_{u,\lambda})^2 \, dt + \tilde C \int_{(-\infty, \lambda]} e^{(2^*-2)(\delta-\tau_1)t} (w^+_{v,\lambda})^2\, dt \\
		&+ \nu \alpha \beta C \int_{(-\infty, \lambda]} e^{(2^*-2)(\delta-\tau_1)t} (w_{v,\lambda}^+)^2   \, dt + \nu \alpha \beta C \int_{(-\infty, \lambda]} e^{(2^*-2)(\delta-\tau_1)t} (w_{u,\lambda}^+)^2   \, dt,
	\end{split}
\end{equation}
Let us define $C_t:= \sup_{t \in (-\infty, \lambda]} e^{(2^*-2)(\delta-\tau_1)t} < +\infty$. Rearranging \eqref{eq:MovStepFinal_u+v}, we achieve the following estimate 
\begin{equation}
	\begin{split}\label{eq:MovStepFinal}		
			\int_{(-\infty, \lambda]} [(w^+_{u,\lambda})']^2  \, dt + \int_{(-\infty, \lambda]} [(w^+_{v,\lambda})']^2  \, dt  	\leq &   \left[C_t(\bar C + \nu \alpha \beta C)- (\delta^2-\gamma)\right] \int_{(-\infty, \lambda]} (w^+_{u,\lambda})^2 \, dt  \\
			& + \left[C_t(\tilde C + \nu \alpha \beta C)- (\delta^2-\gamma)\right] \int_{(-\infty, \lambda]} (w^+_{v,\lambda})^2 \, dt.
	\end{split}
\end{equation}
Now, if we fix $\lambda < -M$ sufficiently negative such that
$$C_t < \min \left\{\frac{\delta^2-\gamma}{\bar C + \nu \alpha \beta C}, \frac{\delta^2-\gamma}{\tilde C + \nu \alpha \beta C}\right\},$$
we obtain
\begin{equation}\label{eq:MovStepFinalConc}		
		\int_{(-\infty, \lambda]} [(w^+_{u,\lambda})']^2 \eta_R^2 \, dt + \int_{(-\infty, \lambda]} [(w^+_{v,\lambda})']^2 \eta_R^2 \, dt \leq 0.
\end{equation}
Hence, we just proved that for any $\lambda < - M$ is true that $y_u \leq y_{u,\lambda}$ and $y_v \leq y_{v,\lambda}$. In other words, if we define		
\begin{equation}\nonumber
	\Lambda=\left\{\lambda \in R : y_u \leq y_{u,s}\,\,\,\text{and}\,\,\, y_v \leq y_{v,s}\,\,\,\text{in} \,\,\, (-\infty, s]\,\,\,\text{for all $s\in(-\infty,\lambda]$}\right\}
\end{equation}
we can say $\Lambda \neq 0$. 

To complete the proof it is sufficient to prove that	$\bar \lambda=\sup\,\Lambda  < +\infty$. The rest of the proof is standard, in particular one can follows a similar argument to the one used in the proof of Theorem \ref{thm:main1}. 

\end{proof}

Now, we shall prove that $(y_u,y_v)$ is a synchronized solution, see \eqref{eq:sinchronizedSol}. This follows if and only if there exists a constant $C>0$ such that $$y_v=C y_u\,.$$ 
In the following result we shall exploit some   ideas from \cite{WeiYao}. The main difference in our setting is the role played by the function 
\begin{equation}\label{eq:f}
	f(s)=s^{2^*-2}+\nu \alpha s^{\alpha-2}-1-\nu\beta s^{\alpha},
\end{equation}
whose number of roots depends on $\alpha, n$ and $\nu$.

\begin{prop}\label{prop:synchro}
	Let $(y_u,y_v)$ be a solution of \eqref{eq:RadialCritProbBis}, then there exists $\tilde{C}>0$ such that $y_u=\tilde C y_v$.
\end{prop}

\begin{proof}
%
%
Let $\tilde{C}$ be a positive constant, then the couple 
\begin{equation}\label{eq:tildeyuyv}
(\tilde{y}_u(t),\tilde{y}_v(t)):=\left(\frac{y_u}{\tilde{C}}(t),y_v(t) \right)
\end{equation} 
satisfies the system 
\begin{equation}
	\begin{cases}\label{eq:RadialCritProbBisC}
		\tilde{y}''_u-(\delta^2-\gamma) \tilde{y}_u + \tilde{C}^{2^*-2}\tilde{y}_u^{2^*-1} + \nu \alpha \tilde{C}^{\alpha-2} \tilde{y}_u^{\alpha-1} \tilde{y}_v^\beta=0 & \text{in}\quad \R \\
		\tilde{y}''_v-(\delta^2-\gamma) \tilde{y}_v + \tilde{y}_v^{2^*-1} + \nu \beta \tilde{C}^{\alpha}\tilde{y}_u^\alpha \tilde{y}_v^{\beta-1}=0 & \text{in}\quad \R \\ 
		\tilde{y}_u,\tilde{y}_v>0 &\text{in}\quad\R.
	\end{cases}
\end{equation}

As a first step, multiplying the first equation of \eqref{eq:RadialCritProbBisC} by $y_v$, one obtains
\begin{equation}\label{proof:synchro1}
(\tilde{y}_u' \tilde{y}_v)'-\tilde{y}'_u \tilde{y}'_v-(\delta^2-\gamma) \tilde{y}_u \tilde{y}_v + \tilde{C}^{2^*-2}\tilde{y}_u^{2^*-1} \tilde y_v+ \nu \alpha \tilde{C}^{\alpha-2} \tilde{y}_u^{\alpha-1} \tilde{y}_v^{\beta+1}=0.
\end{equation}

Analogously, multiplying the second equation of \eqref{eq:RadialCritProbBisC} by $y_u$,
\begin{equation}\label{proof:synchro2}
(\tilde{y}_v' \tilde{y}_u)'-\tilde{y}'_u \tilde{y}'_v-(\delta^2-\gamma) \tilde{y}_u \tilde{y}_v + \tilde{y}_v^{2^*-1} \tilde y_u+ \nu \beta \tilde{C}^{\alpha} \tilde{y}_u^{\alpha+1} \tilde{y}_v^{\beta-1}=0.
\end{equation}

Subtracting \eqref{proof:synchro1} by \eqref{proof:synchro2} gives
\begin{equation}\label{proof:synchro3}
\begin{split}
(\tilde{y}_u' \tilde{y}_v-\tilde{y}_v'\tilde{y}_u)' &+ \tilde{C}^{2^*-2}\tilde{y}_u^{2^*-1} \tilde y_v+ \nu \alpha \tilde{C}^{\alpha-2} \tilde{y}_u^{\alpha-1} \tilde{y}_v^{\beta+1}- \tilde{y}_v^{2^*-1} \tilde y_u- \nu \beta \tilde{C}^{\alpha} \tilde{y}_u^{\alpha+1} \tilde{y}_v^{\beta-1}= \\
(\tilde{y}_u' \tilde{y}_v-\tilde{y}_v'\tilde{y}_u)'&+\tilde{y}_v^{2^*-1}\tilde{y}_u f \left(\tilde{C} \frac{\tilde{y}_u}{\tilde{y}_v}\right)=0,
\end{split}
\end{equation}
where the function $f :(0,+\infty)\to\mathbb{R}$ is defined by \eqref{eq:f}.

By Lemma \ref{lemasymptot}, we recall that $\tilde{y}_u(t),\tilde{y}_v(t) \to 0$ as $t \to -\infty$, $\tilde{y}_u(t),\tilde{y}_v(t)\to 0$ as $t\to +\infty$, and by Lemma \ref{lem:extremums} we have $\tilde{y}'_u(0)=\tilde{y}'_v(0)=0$, assuming by translation that $0$ is the global maximum of $\tilde{y}_u,\tilde{y}_v$. By integrating \eqref{proof:synchro3}
$$
\int_{-\infty}^{0} y_v^{2^*-1}y_u f\left(\tilde{C} \frac{\tilde{y}_u}{\tilde{y}_v} \right) \, dt =0, \qquad \int_{0}^{+\infty} \tilde{y}_v^{2^*-1}\tilde{y}_u f\left(\tilde{C} \frac{\tilde{y}_u}{\tilde{y}_v} \right) \, dt=0.
$$ 

If we deduce that $f\left(\tilde{C} \frac{y_u}{y_v}\right)\leq 0$ or $f\left(\tilde{C} \frac{y_u}{y_v}\right)\geq 0$ for every  $t\in (-\infty,0)$, then it would follows $f\left(\tilde{C} \frac{y_u}{y_v}\right)= 0$ in $ (-\infty,0)$.

To achieve this claim, we note that, by Lemma~\ref{lemasymptot} and in particular by \eqref{eq:quotient}, we have
$$\lim_{t\to -\infty}\frac{y_u(t)}{y_v(t)}=L \in (0,+\infty).$$
We have to study two different cases.

{\bf Case 1: $f(L) \neq 0$.} Without loss of generality we assume that
\begin{equation}\label{eq:signassump}
f(L)<0.
\end{equation}
By \eqref{eq:signassump} there exists $\hat t<0$ such that
\begin{equation} \label{eq:fsign}
	f\left(\tilde{C} \frac{\tilde{y}_u(t)}{\tilde{y}_v(t)}\right) < 0 \ \ \forall t \in (-\infty, \hat t) \text{ and } f\left(\tilde{C} \frac{\tilde{y}_u(\hat t)}{\tilde{y}_v(\hat t)}\right) = f\left(\frac{y_u(\hat t)}{y_v(\hat t)}\right) = 0.
\end{equation}
>From now on, we can fix $\tilde C := \frac{y_u(\hat t)}{y_v(\hat t)}$ so that $$f(\tilde C)=0$$ and $\tilde y_u(\hat t) = \tilde y_v(\hat t)$ as in \eqref{eq:tildeyuyv}.
By integrating \eqref{proof:synchro3} in $(-\infty,t)$ for any $t \leq \hat t$, we get
\begin{equation}\label{proof:synchro5}
(\tilde{y}'_u \tilde{y}_v-\tilde{y}'_v \tilde{y}_u)(t)=-\int_{-\infty}^{t} \tilde{y}_v^{2^*-1}\tilde{y}_u f\left(\tilde{C} \frac{\tilde{y}_u}{\tilde{y}_v} \right) \, ds>0.
\end{equation}

Now, let us assume for a while that there exists $0<t_1<\hat t$, such that 
\begin{equation}\label{proof:synchroclaim}
(\tilde{y}'_u \tilde{y}_v-\tilde{y}'_v \tilde{y}_u)(t_1)=0.
\end{equation}

This would immediately imply that
$$
(\tilde{y}'_u \tilde{y}_v-\tilde{y}'_v \tilde{y}_u)(t_1)+\int_{-\infty}^{t_1} \tilde{y}_v^{2^*-1}\tilde{y}_u f\left(\tilde{C} \frac{\tilde{y}_u}{\tilde{y}_v} \right) \, ds =\int_{-\infty}^{t_1} \tilde{y}_v^{2^*-1}\tilde{y}_u f\left(\tilde{C} \frac{\tilde{y}_u}{\tilde{y}_v} \right) \, ds=0.
$$
Then we could conclude that 
$$f\left(\tilde{C} \frac{\tilde{y}_u(t)}{\tilde{y}_v(t)} \right) = f\left(\frac{y_u(t)}{y_v(t)} \right) =0 \ \ \text{ for any } t \in (-\infty, t_1).$$ 
In particular, this fact would immediately give
\begin{equation}\label{contradictionCase1}
	\frac{y_u(t)}{y_v(t)}= \hat C \ \text{ for } \  t\in(-\infty,t_1) \  \text{ such that } \  f(\hat{C})=0.
\end{equation}	
By uniqueness of the limits at $-\infty$ we deduce that $\tilde C = \hat C = L$, which is a contradiction since we assumed $f(L) \neq 0$ in this case.

The rest of the proof is devoted to show \eqref{proof:synchroclaim}. By contradiction, assume that \eqref{proof:synchroclaim} is false. Then, by \eqref{proof:synchro5}, 
\begin{equation}\label{proof:synchronoclaim}
(\tilde{y}'_u \tilde{y}_v-\tilde{y}'_v \tilde{y}_u)(t)>0 \qquad \mbox{ for all } t \in (-\infty, \hat t).
\end{equation}

Now, multiplying the first equation of \eqref{eq:RadialCritProbBisC} by $\tilde{y}'_u$, one gets,
\begin{equation}\label{proof:synchro6}
\left(\frac{(\tilde{y}'_u)^2}{2}\right)'-\frac{(\delta^2-\gamma)}{2} (\tilde{y}_u^2)'+\frac{\tilde{C}^{2^*-2}}{2^*}(\tilde{y}_u^{2^*})'+\tilde{C}^{\alpha-2}\nu \alpha \tilde{y}_u^{\alpha-1} \tilde{y}_v^\beta \tilde{y}_u'=0.
\end{equation}
Similarly, multiplying the second equation of \eqref{eq:RadialCritProbBisC} by $\tilde{y}'_v$,
 \begin{equation}\label{proof:synchro7}
\left(\frac{(\tilde{y}'_v)^2}{2}\right)'-\frac{(\delta^2-\gamma)}{2} (\tilde{y}_v^2)'+\frac{1}{2^*}(\tilde{y}_v^{2^*})'+\tilde{C}^{\alpha}\nu \beta  \tilde{y}_u^\alpha\tilde{y}_v^{\beta -1} \tilde{y}_v'=0.
\end{equation}

Now, subtracting \eqref{proof:synchro6} and \eqref{proof:synchro7}, integrating in $(-\infty, \hat t)$, and using the fact that 
$$\tilde y_u(\hat t) = \tilde y_v(\hat t),$$
we have
\begin{equation}\label{proof:synchro8}
\begin{split}
& \left(\frac{(\tilde{y}_u')^2}{2} - \frac{(\tilde{y}_v')^2}{2}\right)(\hat t) \\
&\qquad \ \ \quad  + \int_{-\infty}^{\hat t} \left(\frac{\tilde{C}^{2^*-2}}{2^*}(\tilde{y}_u^{2^*})' - \frac{1}{2^*}(\tilde{y}_v^{2^*})' +\tilde{C}^{\alpha-2}\nu \alpha \tilde{y}_u^{\alpha-1} \tilde{y}_v^\beta \tilde{y}_u' - \tilde{C}^{\alpha} \nu \beta \tilde{y}_u^\alpha\tilde{y}_v^{\beta -1} \tilde{y}_v'\, \right)ds=0.
\end{split}
\end{equation}
We note that, by \eqref{eq:fsign}, it follows
\begin{equation}\label{eq:relationcoupligterm}
\int_{-\infty}^{\hat t} (\tilde{y}_u^{2^*})' \, ds = \int_{-\infty}^{\hat t} (\tilde{y}_v^{2^*})' \, ds = \int_{-\infty}^{\hat t} (\tilde{y}_u^{\alpha}\tilde{y}_v^{\beta})' \, ds = \int_{-\infty}^{\hat t} \left( \alpha \tilde{y}_u^{\alpha-1} \tilde{y}_v^{\beta} \tilde y_u' + \beta \tilde{y}_u^{\alpha} \tilde{y}_v^{\beta-1} \tilde y_v' \right) \, ds.
\end{equation}
Using \eqref{eq:relationcoupligterm} in \eqref{proof:synchro8}, we get
 \begin{equation}\label{proof:synchro9}
 	\begin{split}
 		\left(\frac{(\tilde{y}_u')^2}{2} - \frac{(\tilde{y}_v')^2}{2}\right)(\hat t) &+ \int_{-\infty}^{\hat t} \alpha \left(\frac{\tilde{C}^{2^*-2}-1}{2^*} + \nu \tilde C^{\alpha-2}\right) \tilde{y}_u^{\alpha-1} \tilde{y}_v^\beta \tilde{y}_u'  \, ds \\
 		&+ \int_{-\infty}^{\hat t} \beta \left(\frac{\tilde{C}^{2^*-2}-1}{2^*}- \nu  \tilde{C}^{\alpha} \right)  \tilde{y}_u^\alpha\tilde{y}_v^{\beta -1} \tilde{y}_v'\, ds=0.
 	\end{split}
\end{equation}
Recalling that $f(\tilde C)=0$, we deduce that
$$\tilde{C}^{2^*-2}-1 = - \nu  \alpha \tilde{C}^{\alpha-2} + \nu \beta \tilde{C}^{\alpha}.$$
Hence, using the fact that $\alpha + \beta = 2^*$ one can easily check that
$$\alpha \left( \frac{\tilde{C}^{2^*-2}-1}{2^*} + \nu \tilde C^{\alpha-2} \right)= - \beta \left( \frac{\tilde{C}^{2^*-2}-1}{2^*}- \nu  \tilde{C}^{\alpha} \right) = \frac{\nu \alpha \beta}{2^*} \tilde C^{\alpha-2}(1+\tilde C^2).$$
Thanks to this information, equation \eqref{proof:synchro9} becomes
 \begin{equation}\label{proof:synchro10}
	\begin{split}
		\left(\frac{(\tilde{y}_u')^2}{2} - \frac{(\tilde{y}_v')^2}{2}\right)(\hat t) + \alpha \left(\frac{\tilde{C}^{2^*-2}-1}{2^*} + \nu \tilde C^{\alpha-2}\right) \int_{-\infty}^{\hat t}  \tilde{y}_u^{\alpha-1} \tilde{y}_v^{\beta -1} (\tilde{y}_u' \tilde{y}_v  -  \tilde{y}_v' \tilde{y}_u )\, ds=0.
	\end{split}
\end{equation}
By \eqref{proof:synchronoclaim} and the fact that $\tilde y_u(\hat t) = \tilde y_v(\hat t)$, we know that
\begin{equation} \label{eq:contra}
	(\tilde{y}'_u \tilde{y}_v-\tilde{y}'_v \tilde{y}_u)(\hat t) = \tilde{y}_u(\hat t) [\tilde{y}'_u (\hat t)-\tilde{y}'_v (\hat t)]>0.
\end{equation} 
Finally, we arrive at the conclusion that \eqref{proof:synchro8} gives us a contradiction, since \eqref{proof:synchronoclaim} and \eqref{eq:contra} should imply the left hand to be strictly positive.

{\bf Case 2: $f(L) = 0$.} Since $\lim_{t\to -\infty}\frac{y_u(t)}{y_v(t)}=L \in (0,+\infty)$, then we  claim that 
\begin{equation}\label{eq:orderedsolutions}
\exists \ \tilde t < 0 \text{ such that } \frac{y_u(t)}{y_v(t)} \leq L \text{ or } \frac{y_u(t)}{y_v(t)} \geq L \text{ for any } t \in (-\infty , \tilde t].
\end{equation}
To prove this, we argue by contradiction assuming that there exists a sequence of points $\{t_k\}_{k \in \N} \subset (-\infty, \tilde t]$ such that 
$$\frac{y_u(t_k)}{y_v(t_k)} = L \text{ for any } k \in \N \text{ and } f\left(\frac{y_u(t)}{y_v(t)}\right)\neq 0  \text{ for any } t \in (t_k,t_{k+1}).$$
Let us consider any closed interval $[t_k,t_{k+1}]$ and fix $\tilde C = L$. Hence, recalling \eqref{eq:tildeyuyv}, we have that $f\left(\tilde C \frac{\tilde y_u(t)}{\tilde y_v(t)}\right) = f\left(\frac{y_u(t)}{y_v(t)}\right)\neq 0$ for any $t \in (t_k,t_{k+1})$. Let us consider 
$$w(t) = \tilde y_u(t) - \tilde y_v(t).$$
Without loss of generality we can suppose that $w(t)>0$ for any $t \in (t_k,t_{k+1})$ and $w(t_k)=w(t_{k+1})=0$. We point out that $(\tilde y_u, \tilde y_v)$ weakly solves \eqref{eq:RadialCritProbBisC}, i.e.
\begin{equation}\label{eq:weakcontra}
\begin{split}
	\int_\R \tilde{y}'_u \varphi' \, dt &= - (\delta^2-\gamma) \int_\R  \tilde{y}_u \varphi \, dt + \int_\R \left(  \tilde{C}^{2^*-2} \tilde{y}_u^{2^*-1} + \nu \alpha \tilde{C}^{\alpha-2} \tilde{y}_u^{\alpha-1} \tilde{y}_v^\beta \right) \varphi \,dt,  \quad \forall \varphi \in C^1_c(\R), \\
	\\
	\int_\R \tilde{y}'_v \varphi' \, dt &= -(\delta^2-\gamma) \int_\R \tilde{y}_v \varphi \, dt+ \int_\R \left(\tilde{y}_v^{2^*-1} + \nu \beta \tilde{C}^{\alpha}\tilde{y}_u^\alpha \tilde{y}_v^{\beta-1} \right) \varphi \, dt, \quad \forall \varphi \in C^1_c(\R).
\end{split}
\end{equation}
We consider
$$\varphi:= w(t) \chi_{(t_k,t_{k+1})} (t)$$
as test function in \eqref{eq:weakcontra} so that, subtracting both the equations, we get
\begin{equation}\label{eq:weakcontra2}
	\begin{split}
		\int_{(t_{k+1},t_k)} (w')^2  \, dt =& - (\delta^2-\gamma) \int_{(t_{k+1},t_k)}  w^2 \, dt \\
		&+ \int_{(t_{k+1},t_k)} \left(  \tilde{C}^{2^*-2} \tilde{y}_u^{2^*-1} + \nu \alpha \tilde{C}^{\alpha-2} \tilde{y}_u^{\alpha-1} \tilde{y}_v^\beta - \tilde{y}_v^{2^*-1}- \nu \beta \tilde{C}^{\alpha}\tilde{y}_u^\alpha \tilde{y}_v^{\beta-1} \right) w \,dt\\
		=& - (\delta^2-\gamma) \int_{(t_{k+1},t_k)}  w^2 \, dt + \int_{(t_{k+1},t_k)} \tilde y_v^{2^*-1}\left[ g\left(\frac{\tilde y_u}{\tilde y_v}\right) - g(1) \right] w \,dt,
	\end{split}
\end{equation}
where $g: (0,+\infty) \rightarrow \R$ is defined as 
$$g(s):= \tilde C^{2^*-2} s^{2^*-1}+\nu \alpha \tilde C^{\alpha-2} s^{\alpha-1}- \nu\beta \tilde C^\alpha s^{\alpha} \quad \text{and} \quad g(1)=C^{2^*-2} +\nu \alpha \tilde C^{\alpha-2} - \nu\beta \tilde C^\alpha = 1,$$
since $0=f(\tilde C)=g(1)-1$. By the Mean Value's Theorem we deduce that there exists $l \in \left(1,\frac{\tilde y_u}{\tilde y_v}\right)$ such that \eqref{eq:weakcontra2} becomes
\begin{equation}\label{eq:weakcontra3}
	\begin{split}
		\int_{(t_{k+1},t_k)} (w')^2  \, dt \leq  - (\delta^2-\gamma) \int_{(t_{k+1},t_k)}  w^2 \, dt + \int_{(t_{k+1},t_k)} \tilde y_v^{2^*-2} |g'(l)| w^2 \,dt.
	\end{split}
\end{equation}
Now we observe that $|g'(l)|\leq C$. Since $\frac{\tilde y_u(t)}{\tilde y_v(t)} \rightarrow 1$ as $t \rightarrow - \infty$, up to redefine $\tilde t$ we can assume that
$$\tilde y_v (t) \leq \left(\frac{\delta^2-\gamma}{C}\right)^{\frac{1}{2^*-2}}$$
so that, by \eqref{eq:weakcontra3}, we get
\begin{equation}\label{eq:weakcontra4}
	\begin{split}
		\int_{(t_{k+1},t_k)} (w')^2  \, dt \leq  0,
	\end{split}
\end{equation}
which implies that $w'(t)=0$ for any $t \in (t_{k+1},t_k)$ and hence $w(t)=C$. But, since $w(t_{k+1})=w(t_k)=0$ we deduce that $w(t) = 0$ for any $t \in (t_{k+1},t_k)$. As consequence we deduce $\tilde y_u(t)= \tilde y_v(t)$ for any $t \in (t_{k+1},t_k)$ providing a contradiction. \\
Thus  $f\left(\frac{y_u(t)}{y_v(t)} \right)$ has a sign at $-\infty$. Hence,  we may run over the arguments of \textbf{Case 1} recovering \eqref{contradictionCase1}. 
Therefore we have that 
$$y_u(t)=\tilde C y_v(t) \text{ for any } t \in \R,$$
with $\tilde C=L$ or, if this is not the case, the argument can be repeated proving the thesis.
\end{proof}

At this point, we are able to prove the classification result.

\begin{proof}[Proof of Theorem \ref{thm:main2}]

Let $(u,v) \in D^{1,2}(\R^n) \times D^{1,2}(\R^n)$ be a solution to \eqref{eq:doublecriticalsingularequbdd}. By Proposition~\ref{prop:synchro}, one gets 
$$y_u = \tilde C y_v.$$
This gives that
$$ u = \tilde C v.$$
By direct computations, it follows that $(u,v)$ solves
\begin{equation}\label{eq:systemconstantC}
\begin{cases}
	\displaystyle -\Delta u \,=\gamma \frac{u}{|x|^2} + \left(1+ \frac{\nu \alpha}{\tilde C^\beta}\right) u^{2^*-1} & \text{in}\quad\R^n \vspace{0.2cm}\\
	\displaystyle -\Delta v \,=\gamma \frac{v}{|x|^2} + \left(1+ \nu \beta \tilde C^\alpha \right) v^{2^*-1} & \text{in}\quad\R^n\\
	u,v> 0 &  \text{in}\quad\R^n \setminus \{0\}.
\end{cases}
\end{equation}
By dilatation and rescaling, exploiting \cite{terracini}, it follows that
$$(u,v)=(c_1 \mathcal{U}_{\mu_0}, c_2 \mathcal{U}_{\tilde \mu})$$
with $\mathcal{U}_\mu$ given by the expression \eqref{eq:terraciniane} and $\tilde \mu, \mu_0>0$. Moreover, by the proportionality of $u$ and $v$, one can deduce that $\tilde \mu = \mu_0$. Since $\tilde C$ satisfies $f(\tilde C)=0$, then 
$$
1+ \frac{\nu \alpha}{ \tilde C^\beta}=\frac{1}{\tilde C^{2^*-2}}+\frac{\nu \beta}{\tilde C^{\beta-2}}.
$$
Introducing such information in \eqref{eq:systemconstantC}, we can conclude that
$$
\displaystyle -\Delta \mathcal U_{\mu_0} \,=\gamma \frac{\mathcal U_{\mu_0}}{|x|^2} + \left(\frac{1}{\tilde C^{2^*-2}}+\frac{\nu \beta}{\tilde C^{\beta-2}}\right) c_1^{2^*-2}\mathcal U_{\mu_0}^{2^*-1}= \gamma \frac{\mathcal U_{\mu_0}}{|x|^2} + \left(1+ \nu \beta \tilde C^\alpha\right) c_2^{2^*-2}\mathcal U_{\mu_0}^{2^*-1} \quad  \text{in } \R^n,
$$
which immediately implies that $\tilde{C}=\frac{c_1}{c_2}$ and, moreover, $(c_1,c_2)$ solution to the system \eqref{eq:systemconstants}, which completes the proof.
\end{proof}

\bigskip

\end{document}